\documentclass[10pt,a4paper]{amsart}
\usepackage[utf8]{inputenc}
\usepackage{amsmath}
\usepackage{amsfonts}
\usepackage{amssymb}
\usepackage{amsthm}
\usepackage{color}

\newtheorem{theorem}{Theorem}[section]
\newtheorem{corollary}[theorem]{Corollary}

\newtheorem{lemma}[theorem]{Lemma}
\newtheorem{proposition}[theorem]{Proposition}

\theoremstyle{definition}

\newcommand{\R}{\mathbb{R}}
\newcommand{\Z}{\mathbb{Z}}
\newcommand{\N}{\mathbb{N}}

\calclayout

\makeatletter
\@namedef{subjclassname@2020}{%
  \textup{2020} Mathematics Subject Classification}
\makeatother

\begin{document}
\title[Smoothing estimates in modulation spaces]{On smoothing estimates in modulation spaces and the nonlinear Schr\"odinger equation with slowly decaying initial data}
\author{Robert Schippa}
\email{robert.schippa@kit.edu}

\address{Department of Mathematics, Karlsruhe Institute of Technology, Englerstrasse 2, 76131 Karlsruhe, Germany}

\begin{abstract} We show new local $L^p$-smoothing estimates for the Schr\"odinger equation in modulation spaces via decoupling inequalities. Furthermore, we probe necessary conditions by Knapp-type examples for space-time estimates of solutions with initial data in modulation and $L^p$-spaces. The examples show sharpness of the smoothing estimates up to the endpoint regularity in a certain range. Moreover, the examples rule out global Strichartz estimates for initial data in $L^p(\R^d)$ for $d \geq 1$ and $p > 2$, which was previously known for $d \geq 2$. The estimates are applied to show new local and global well-posedness results for the cubic nonlinear Schr\"odinger equation on the line. Lastly, we show $\ell^2$-decoupling inequalities for variable-coefficient versions of elliptic and non-elliptic Schr\"odinger phase functions.
\end{abstract}
\maketitle

\section{Introduction}
In this article we show new space-time estimates for the Schr\"odinger equation with slowly decaying data outside $L^2$-based Sobolev spaces:
\begin{equation}
\label{eq:NLSIntro}
\left\{
\begin{array}{cl}
i \partial_t u + \Delta u &= 0, \quad (t,x) \in \R \times \R^d, \\
u(0) &= u_0 \in X.
\end{array} \right.
\end{equation}
In the present work we consider modulation spaces $X =M^s_{p,q}(\R^d)$, $2 \leq p < \infty$, which are compared to initial data in $L^p$-based Sobolev spaces $X = L_\alpha^p(\R^d) = \langle D \rangle^{-\alpha} L^p(\R^d)$, $2 \leq p<\infty$. The latter initial data were recently considered by Dodson--Soffer--Spencer\cite{DodsonSofferSpencer2020} (see also \cite{DodsonSofferSpencer2020A}) and by Mandel \cite{Mandel2020}.

 Modulation spaces are likewise used to model slowly decaying initial data. Feichtinger introduced modulation spaces in \cite{Feichtinger1983}. He provided a more recent account with \cite{Feichtinger2006} emphasizing the role of modulation spaces in signal processing; see also the textbook by Gr\"ochenig \cite{Groechenig2001}. As the body of literature is huge, we refer to the PhD thesis of Chaichenets \cite{Chaichenets2018} and references therein for a more exhaustive account on modulation spaces in the context of Schr\"odinger equations.
 
For the definition of modulation spaces, consider the Fourier multipliers
\begin{equation*}
(\Box_k f) \hat (\xi) = \sigma_k(\xi) \hat{f}(\xi), \quad k \in \Z^d,
\end{equation*}
with $(\sigma_k)_{k \in \Z^d} \subseteq C^\infty_c(\R^d)$ a smooth partition of unity, adapted to the translated unit cubes $Q_k = k + [-\frac{1}{2},\frac{1}{2})^d$. For the precise definition, we refer to \cite[Section~2.1]{Chaichenets2018}. We can suppose that $\sigma_k(\xi) = \sigma_0(\xi -k)$. The norm is defined by
\begin{equation*}
\| f \|_{M^s_{p,q}} = \big( \sum_{k \in \Z^d} \langle k \rangle^{qs} \| \Box_k f \|_{L^p(\R^d)}^q \big)^{\frac{1}{q}}.
\end{equation*}
If $s=0$, we write $M_{p,q}$. Modulation spaces are closely related with $L^p$-spaces. For illustration, we collect embedding properties: By the embedding of $\ell^p$-spaces and Bernstein's inequality, we have 
\begin{align}
\label{eq:EmbeddingModulationI}
M^s_{p,q_1} &\hookrightarrow M^s_{p,q_2} \quad (q_1 \leq q_2), \\
\label{eq:EmbeddingModulationII}
M^s_{p_1,q} &\hookrightarrow M^s_{p_2,q} \quad (p_1 \leq p_2).
\end{align}
Rubio de Francia's inequality (\cite{RubioDeFrancia1985}) and duality yield
\begin{align}
\label{eq:EmbeddingModulationIII}
M_{p,p'} &\hookrightarrow L^p \hookrightarrow M_{p,p} \quad (2 \leq p \leq \infty), \\
\label{eq:EmbeddingModulationIV}
M_{p,p} &\hookrightarrow L^p \hookrightarrow M_{p,p'} \quad (1 \leq p \leq 2).
\end{align}
Furthermore, we can trade regularity for summability as
\begin{equation}
\label{eq:RegularitySummability}
M^{s_1}_{p,q_1}(\R^d) \hookrightarrow M^{s_2}_{p,q_2}(\R^d) \qquad s_1 - s_2 > d \big( \frac{1}{q_2} - \frac{1}{q_1} \big) > 0
\end{equation}
by applying H\"older's inequality (cf. \cite[Proposition~2.31]{Chaichenets2018}). By Plancherel's theorem, $M_{2,2} \sim L^2$. 

Let $U(t) = e^{it \Delta}$ denote the propagator of \eqref{eq:NLSIntro}, and let $Uf$ denote the free solution for $f \in \mathcal{S}'(\R^d)$. Our first result are new smoothing estimates
\begin{equation}
\label{eq:StrichartzModulation}
\| Uf \|_{L^p([-1,1] \times \R^d)} \lesssim \| f \|_{M^s_{p,q}(\R^d)}
\end{equation}
via $\ell^2$-decoupling. We prove the following:
\begin{theorem}
\label{thm:SmoothingModulationSpaces}
Suppose that $d \geq 1$, $p \geq 2$, and $1 \leq q \leq \infty$.
\begin{itemize}
\item[(A)]  If $2 \leq p \leq \frac{2(d+2)}{d}$, then
\eqref{eq:StrichartzModulation} holds true provided that $s > \max \big( 0,\frac{d}{2} - \frac{d}{q} \big)$.
\item[(B)] If $\frac{2(d+2)}{d} \leq p \leq \infty$ and $2 \leq q \leq \infty$, then \eqref{eq:StrichartzModulation} holds true provided that $s > d - \frac{d+2}{p} - \frac{d}{q} $.
\item[(C)] If $\frac{2(d+2)}{d} \leq p \leq \infty$ and $1 \leq q \leq 2$, then \eqref{eq:StrichartzModulation} holds true provided that $s>2 \big(1 - \frac{1}{q} \big) \big( \frac{d}{2} - \frac{d+2}{p} \big)$.
\item[(D)] If $q =1$, then \eqref{eq:StrichartzModulation} holds true with $s = 0$.
\item[(E)] If $d=1$, $p=4$, and $q=2$, then we find \eqref{eq:StrichartzModulation} to hold for $s=0$.
\end{itemize}
\end{theorem}
The key argument in the proof of the estimates for $q \geq 2$ are the $\ell^2$-decoupling inequalities for the paraboloid due to Bourgain--Demeter (cf. \cite{BourgainDemeter2015}). It turns out that after localization in space and parabolic rescaling the $\ell^2$-decoupling inequality yields Strichartz estimates in modulation spaces by a kernel estimate. For the proof of (E), we use as substitute for decoupling a reverse square function estimate, which is known to hold without $\varepsilon$-loss. Originally, Wolff \cite{Wolff2000} brought up decoupling for the cone to analyze $L^p$-smoothing estimates for the (half-)wave equation:
\begin{equation*}
\| e^{it \sqrt{- \Delta}} f \|_{L^p([1,2] \times \R^d)} \lesssim \| f \|_{L^p_s(\R^d)}.
\end{equation*}
We refer to \cite{MockenhauptSeegerSogge1992,MockenhauptSeegerSogge1993} and \cite[Chapter~8]{Sogge2017} for further reading.

Regarding Strichartz estimates in modulation spaces, it seems the above space-time estimates were previously not investigated in the literature. Space-time estimates with an additional window decomposition were shown by B. Wang \emph{et al.} \cite{WangHudzik2007,WangHuang2007,WangZhaoBoling2006}; see also Zhang \cite{Zhang2013}. We also refer to the surveys by Wang--Huo--Hao--Guo \cite{WangHuoHaoGuo2011} and Ruzhansky--Sugimoto--Wang \cite{RuzhanskySugimotoWang2012}. These estimates were further applied to prove well-posedness for nonlinear equations. For context we refer to S. Guo's work \cite{Guo2017}, in which he proved local well-posedness of the NLS in modulation spaces $M_{2,p}$ for $2 < p < \infty$. Oh--Wang \cite{OhWang2020} globalized this using the complete integrability. Below we discuss well-posedness of the cubic NLS outside $L^2$-based Sobolev spaces in greater detail.

We shall also compare Strichartz estimates in modulation spaces with $L^p$-smoothing estimates for Schr\"odinger equations, which were first discussed by Rogers \cite{Rogers2008}:
\begin{equation}
\label{eq:SmoothingEstimates}
\| U f \|_{L^p(I \times \R^d)} \lesssim \| f \|_{L_\alpha^p(\R^d)}
\end{equation}
 Rogers showed that the validity of \eqref{eq:SmoothingEstimates} for some $\alpha$ is equivalent to validity of the adjoint Fourier restriction estimate $R^*(p \to p)$. We refer to \cite{RogersSeeger2010, LeeRogersSeeger2013} for further discussion.
 The decoupling inequality can serve as a common base for \eqref{eq:StrichartzModulation} and \eqref{eq:SmoothingEstimates}.
 
 \medskip

Secondly, we give new necessary conditions for estimates of the kinds
\begin{equation}
\label{eq:StrichartzModulationCounter}
\| Uf \|_{L^p([-1,1] \times \R^d)} \lesssim \| f \|_{M^s_{p,q}(\R^d)}
\end{equation}
and
\begin{equation}
\label{eq:StrichartzLp}
\| U f \|_{L_t^p(I;L_x^q(\R^d))} \lesssim \| f \|_{L_s^r(\R^d)}
\end{equation}
for $I \in \{ [-1,1], \R \}$. We prove the following necessary conditions:

\begin{proposition}
\label{prop:NecessaryConditions}
Let $p \geq 2$. Necessary for \eqref{eq:StrichartzModulationCounter} to hold true is
\begin{equation}
\label{eq:DerivativeModulationSmoothing}
s \geq \max \big( 0,d - \frac{d+2}{p} - \frac{d}{q} \big).
\end{equation}
Necessary for \eqref{eq:StrichartzLp} to hold for $I=[-1,1]$ is
\begin{equation}
\label{eq:ConditionsDerivative}
s \geq \max \big( 0, d- \frac{d}{q} - \frac{2}{p} - \frac{d}{r} \big), \quad q \geq r.
\end{equation}
If $I = \R$, we have the additional conditions
\begin{equation}
\label{eq:GlobalIntegrabilityConditions}
\frac{2}{p} + \frac{d}{q} \leq \frac{d}{r}.
\end{equation}
\end{proposition}
We refer to Section \ref{section:NecessaryConditions} for the discussion of further conditions.

\medskip

Proposition \ref{prop:NecessaryConditions} shows that the estimates in Theorem \ref{thm:SmoothingModulationSpaces} are sharp up to the endpoint regularity for $2 \leq p \leq \frac{2(d+2)}{d}$ and for ($\frac{2(d+2)}{d} \leq p \leq \infty$ and $2 \leq q \leq \infty$). Moreover, the examined examples show that global estimates \eqref{eq:StrichartzLp} for $I= \R$ and $s = 0$ are impossible for $p > 2$. Mandel \cite{Mandel2020} previously showed this for $d \geq 2$ with a more explicit example.

\begin{corollary}
\label{cor:NonexistenceStrichartz}
Suppose that $d \geq 1$, $p,q \in [1,\infty]$, and $r \in (2,\infty]$. Then, there is no $C$ such that the estimate
\begin{equation}
\label{eq:StrichartzEstimate}
\| U f \|_{L_t^p(\R,L_x^q(\R^d))} \leq C \| f \|_{L^r(\R^d)}
\end{equation}
holds true for any $f \in L_{rad}^r(\R^d)$.
\end{corollary}

A major difference to $L^p$-based Sobolev spaces, $p \neq 2$, is that the propagator $U(t)$ is bounded on modulation spaces. Already by B\'{e}nyi \emph{et al.} \cite{BenyiGroechenigOkoudjouRogers2007} was proved the bound for $2 \leq p \leq \infty$:
\begin{equation}
\label{eq:FixedTimeEstimateModulationSpaces}
\| U(t) \|_{M^s_{p,q}(\R^d) \to M^s_{p,q}(\R^d)} \lesssim \langle t \rangle^{d \big| \frac{1}{2} - \frac{1}{p} \big|}.
\end{equation}
Chaichenets \cite[Section~3.1]{Chaichenets2018} observed by duality that the same bound holds for $1 \leq p \leq 2$, and sharpness was initially shown by Cordero--Nicola \cite{CorderoNicola2009}. In \cite[Section~3.1]{Chaichenets2018} Gaussians were used as window functions in the definition of modulation spaces. Then the sharpness follows by computation of the kernel for a window. In Section \ref{section:NecessaryConditions} we recover the sharpness of \eqref{eq:FixedTimeEstimateModulationSpaces} by the discussed examples.

\begin{corollary}[{\cite[Section~3.1]{Chaichenets2018}}]
\label{cor:FixedTimeEstimateModulation}
The estimate \eqref{eq:FixedTimeEstimateModulationSpaces} is sharp for $1 \leq p,q \leq \infty$ and any $s,t \in \R$.
\end{corollary}

 We think that this gives a more robust proof of sharpness. In fact, all the results proved in Sections \ref{section:NecessaryConditions} and \ref{section:DecouplingStrichartzEstimates} have straight-forward counterparts for fractional Schr\"odinger equations
\begin{equation}
\label{eq:FEQ}
\left\{ \begin{array}{cl}
i \partial_t u + (-\Delta)^{\alpha/2} u &= 0, \quad (t,x) \in \R \times \R^d, \\
u(0) &= u_0 \in M^s_{p,q}(\R^d) \end{array} \right.
\end{equation}
for $\alpha > 1$ or non-elliptic Schr\"odinger equations, e.g.,
\begin{equation}
\label{eq:NonellipticEQ}
\left\{ \begin{array}{cl}
i \partial_t u + (\partial_1^2 - \partial_2^2) u &= 0, \quad (t,x) \in \R \times \R^2, \\
u(0) &= u_0 \in M^s_{p,q}(\R^2). \end{array} \right.
\end{equation}
The latter follow by considering the decoupling inequalities from \cite{BourgainDemeter2017GeneralDecoupling}.
Also generalizations to variable coefficients seem possible under additional assumptions. 
As decoupling inequalities for variable coefficient versions of \eqref{eq:NonellipticEQ} are not explicit in the literature, we take the opportunity to prove them here. As the results are more technical to state, we refer to Subsection \ref{subsection:IntroductionVariableCoefficient}. Section \ref{section:VariableCoefficientDecoupling} is based on Chapter~7 of the author's PhD thesis \cite{Schippa2019PhDThesis}.

We point out that the examples from Section \ref{section:NecessaryConditions} rely on (non-)stationary phase estimates, for which the precise form of the phase function is not important, as long as the characteristic surface has non-vanishing Gaussian curvature.

We apply the Strichartz estimates to the cubic nonlinear Schr\"odinger equation:
\begin{equation}
\label{eq:NLSIntroduction}
\left\{ \begin{array}{cl}
i \partial_t u + \Delta u &= |u|^2 u \quad (t,x) \in \R \times \R, \\
u(0) &= f \in D,
\end{array} \right.
\end{equation}
where $D \in \{ L^p_\alpha(\R), M^s_{p,2}(\R) + L^2(\R) \}$. We prefer to work in these slightly larger spaces as the Duhamel term is in $L^2(\R)$. Denote the solution space by $S$. For the proof of local well-posedness we apply the $L^p$-smoothing estimates to the homogeneous equation and use the usual $L^2$-based inhomogeneous Strichartz estimates for the Duhamel integral. We shall see that using modulation spaces allows to save Sobolev regularity:  \eqref{eq:NLSIntroduction} has a local solution if $f \in M^{s}_{6,2}$ for any $s>0$, but considering $f \in L^6_{s}(\R)$ requires $s > \frac{1}{3}$. Moreover, our local results hold likewise for the focusing equation, see below.

\begin{theorem}
\label{thm:LWPL4L6}
Let $\varepsilon>0$. \eqref{eq:NLSIntroduction} is analytically well-posed in the spaces $D \in \{ L^4_\varepsilon(\R), M_{4,2}(\R) \}$, $S_T= L_t^{\frac{24}{7}}([0,T],L^4(\R))$, and in the spaces $D \in \{ L^6_{\frac{1}{3}+\varepsilon}(\R), M^{\varepsilon}_{6,2}(\R) + L^2(\R) \}$, $S_T= L_t^3([0,T],L^6(\R))$, i.e., there is $T = T(\| f \|_{D})$ such that there is a unique solution in $S_T$.

 Furthermore, for $D \in \{ M_{4,2}(\R) , M^{\varepsilon}_{6,2}(\R) + L^2(\R) \}$, we have $u \in C([0,T], D+ L^2(\R))$ with continuous dependence on the initial data.
\end{theorem}
We remark that one could also add regularity in $L^p$-spaces to bound the propagator in $L_x^p(\R)$ via fixed-time estimates in order to obtain continuous curves in $L^p(\R) + L^2(\R)$. However, one always leaves the space of initial values as the Schr\"odinger propagator is unbounded in $L^p$ for $p \neq 2$.

\medskip

As the NLS is one of the most prominent nonlinear dispersive equations, the body of literature on its well-posedness is vast. To put our results into context, we only mention few results and also refer to the references therein. Tsutsumi \cite{Tsutsumi1987} applied classical $L^2$-based Strichartz estimates to prove global well-posedness in $L^2(\R)$. This is the limit of analytic well-posedness \cite{ChristCollianderTao2003,KenigPonceVega2001} in $L^2$-based Sobolev spaces. Recently, Harrop-Griffiths--Killip--Vi\c{s}an \cite{HarropGriffithsKillipVisan2020} proved sharp global well-posedness in $H^s(\R)$, $s>-1/2$, using complete integrability. Outside $L^2$-based Sobolev spaces, we mention the early works by Vargas--Vega \cite{VargasVega2001} and Gr\"unrock \cite{Gruenrock2005} in Fourier Lebesgue spaces. Hyakuna \cite{Hyakuna2020} proved well-posedness results in $L^p$-spaces for some $1<p<2$, and Correia \cite{Correia2018} considered generalized energy spaces $\dot{H}^1(\R^d) \cap L^p(\R^d)$ for $p > 2$. Chaichenets \emph{et al.} \cite{Chaichenets2017} showed the first global results in modulation spaces $M_{p,p'}$ for $p$ sufficiently close to $2$ without smallness assumption on the initial data; see also \cite{Chaichenets2019,Chaichenets2020}. Bourgain \cite{Bourgain1993A} showed how linear Strichartz estimates combined with Galilean invariance and orthogonality arguments yield low regularity well-posedness in the periodic case (see also \cite{Bourgain1993B} for KdV-like equations). We believe that the arguments extend to modulation spaces.

\medskip

Furthermore, Kenig--Ponce--Vega \cite{KenigPonceVega2001} pointed out that the Galilean invariance of the nonlinear Schr\"o\-din\-ger equation 
\begin{equation*}
u(t,x) \to e^{iNx} e^{it N^2} u(t,x-2tN)
\end{equation*}
 indicates mild ill-posedness for negative Sobolev regularity: Suppose that the existence time of the solution to $u_0 \in M^{s}_{p,q}(\R)$ with $s<0$ and $p,q<\infty$ is given by $T=T(\| u_0 \|_{M^s_{p,q}})$. As $N \to \infty$, we find $\| e^{iN \cdot} u(0,\cdot) \|_{M^s_{p,q}} \to 0$, but $T$ is invariant. The restriction on $p,q$ is required to assume that $\mathcal{S}(\R) \hookrightarrow M^s_{p,q}(\R)$ is dense. This argument can be made rigid in the focusing case:
 \begin{equation}
 \label{eq:Focusing}
 \left\{ \begin{array}{cl}
  i \partial_t u + \Delta u &= -|u|^2 u, \quad (t,x) \in \R \times \R, \\
 u(0) &= u_0 \in M^s_{p,q}(\R).
 \end{array} \right.
 \end{equation}
As in \cite{KenigPonceVega2001}, we consider the two-parameter family of traveling wave solutions derived from the stationary solution by scaling and Galilean invariance:
\begin{equation*}
u_{\omega,N}(x,t) = e^{-it (N^2 - \omega^2)} e^{iNx} f_\omega(x-2tN), \quad f_\omega(x) = \omega f(\omega x), \quad f(x) = \sqrt{2} \text{sech}(x).
\end{equation*}
By this, \eqref{eq:Focusing} is not locally uniformly well-posed in $M^s_{p,q}(\R)$ for $s<0$. Thus, the corresponding result of Theorem \ref{thm:LWPL4L6} for the focusing equation covers analytic well-posedness up to the endpoint. We believe that this is also true for the defocusing equation as shown by Christ--Colliander--Tao \cite{ChristCollianderTao2003} in Sobolev spaces of negative regularity. Details on mild ill-posedness will be given elsewhere. Well-posedness with mere continuous dependence can still be possible due to complete integrability (cf. \cite{HarropGriffithsKillipVisan2020}). For more ill-posedness results, i.e., norm inflation and infinite loss of regularity, we refer to Bhimani--Carles \cite{BhimaniCarles2020} and references therein.

\medskip

In the recent work \cite{DodsonSofferSpencer2020} Dodson--Soffer--Spencer (see also \cite{DodsonSofferSpencer2020A}) used fixed-time $L^p$-estimates
\begin{equation}
\label{eq:FixedTimeEstimate}
\| e^{it \Delta} f \|_{L^p(\R^d)} \lesssim \| f \|_{L^p_\alpha(\R^d)}
\end{equation}
and Picard iteration to prove well-posedness of \eqref{eq:NLSIntroduction} with $X_0 = L^p_s(\R)$ for $2<p<\infty$. We remark that the sharp derivative loss $\alpha = 2d \big| \frac{1}{2} - \frac{1}{p} \big|$ for $L^p$-estimates \eqref{eq:FixedTimeEstimate} is known since the work of Fefferman--Stein \cite{FeffermanStein1972} and Miyachi \cite{Miyachi1981}. However, the sharp estimates were not used in \cite{DodsonSofferSpencer2020}, which resulted in high regularity. By now it is well-understood how the fixed-time $L^p$-estimates for the Schr\"odinger propagator follow by the embeddings \eqref{eq:EmbeddingModulationI}-\eqref{eq:EmbeddingModulationIV} and the invariance of modulation spaces under the propagator. We show how passing through modulation spaces improves the results in \cite{DodsonSofferSpencer2020} for $p=4n+2$, $n \geq 2$, and also show global results with arguments due to Dodson \emph{et al.} \cite{DodsonSofferSpencer2020}. We contend that the splitting method applied in \cite{VargasVega2001,Chaichenets2017}, further reaching back to Bourgain's seminal contribution \cite{Bourgain1998}, is related with the present approach to prove the global result. 

\begin{theorem}
\label{thm:GlobalWellposedness}
Let $T > 0$ and $s > \frac{3}{2}$. If $f \in M_{4,2}^{s}(\R)$, then there exists a unique solution $u \in L_t^{\frac{24}{7}}([0,T],L^4) \cap C([0,T],M_{4,2}^{s}(\R) + L^2(\R))$ to \eqref{eq:CubicNLSAbstract}, which depends continuously on the initial data, i.e., for any $T>0$ and $f_n \to f \in M_{4,2}^s(\R)$, we have
\begin{equation*}
\| u_n - u \|_{L_t^{\frac{24}{7}}([0,T],L^4) \cap C([0,T],M_{4,2}^{s}(\R) + L^2(\R))} \to 0.
\end{equation*}

If $f \in M^s_{6,2}(\R)$, then there exists a unique solution $u \in L_t^3([0,T],L^6) \cap C([0,T],M_{6,2}^{s}(\R) + L^2(\R))$ with continuous data-to-solution mapping $f \mapsto u$ as above.
\end{theorem}

\medskip

\emph{Outline of the paper.} In Section \ref{section:NecessaryConditions} we give necessary conditions for $L^p$-smoothing estimates in modulation spaces and Strichartz estimates in $L^p$-based Sobolev spaces. These rule out global Strichartz estimates for initial data in $L^r(\R^d)$, $r>2$. In Section \ref{section:DecouplingStrichartzEstimates} we show Theorem \ref{thm:SmoothingModulationSpaces} via $\ell^2$-decoupling. In Section \ref{section:LWPNLS} the estimates are applied to show new local and global well-posedness for the NLS. In Section \ref{section:VariableCoefficientDecoupling} we show decoupling inequalities for variable coefficient versions of Schr\"odinger equations.

\section{Necessary conditions}
\label{section:NecessaryConditions}

The purpose of this section is to collect necessary conditions to find the following estimates to hold:
\begin{equation}
\label{eq:ModulationSpaceCounter}
\| U f \|_{L_t^p([-1,1],L^q(\R^d))} \lesssim \| f \|_{M^s_{r,t}(\R^d)},
\end{equation}
and
\begin{equation}
\label{eq:StrichartzCounter}
\| U f \|_{L_t^p(I;L^q_x(\R^d))} \lesssim \| f \|_{L^r_s(\R^d)}.
\end{equation}
We shall use three Knapp-type examples, which are essentially well-known in the literature \cite{Tao2006,RogersSeeger2010}. However, it seems that these have not been examined in the above contexts.

\medskip

We start with the \textit{anisotropic Knapp example at unit frequencies}, which was previously used to determine the range of integrability coefficients for the $L^2$-based Strichartz estimate \eqref{eq:StrichartzCounter} (cf. \cite{Tao2006,Strichartz1977}). Consider
\begin{equation}
\label{eq:AnisotropicKnapp}
\hat{g}_\varepsilon(\xi ) = \chi_{(1-\varepsilon,1+\varepsilon)}(\xi_1) \chi_{(-\varepsilon,\varepsilon)}(\xi_2) \ldots \chi_{(-\varepsilon,\varepsilon)}(\xi_d).
\end{equation}
We compute
\begin{equation}
\label{eq:NormsAnisotropicKnapp}
\| f \|_{L_s^r(\R^d)} \sim \varepsilon^{d- \frac{d}{r}}, \text{ and } \| g_\varepsilon \|_{M^s_{r,t}(\R^d)} \sim \varepsilon^{d - \frac{d}{r}}
\end{equation}
for any $s \in \R$, $1 \leq r,t \leq \infty$.\\
We observe
\begin{equation*}
\begin{split}
U g_\varepsilon(x,t) &= C_d \int_{\R^d} e^{i(x.\xi + t |\xi|^2)} \chi_{(1-\varepsilon,1+\varepsilon)}(\xi_1) \chi_{(-\varepsilon,\varepsilon)}(\xi_2)\ldots \chi_{(-\varepsilon,\varepsilon)}(\xi_d) d\xi \\
&= C_d e^{i x_1} e^{it} \int_{\R^d} e^{i((x+2t e_1).\xi + t |\xi|^2)} \chi_{(-\varepsilon,\varepsilon)}(\xi_1) \ldots \chi_{(-\varepsilon,\varepsilon)}(\xi_d) d\xi.
\end{split}
\end{equation*}
Hence, $|Uf(x,t)| \gtrsim \varepsilon^d$ provided that $t \in [-\varepsilon^{-2},\varepsilon^{-2}]$ and $|x+2t e_1| \lesssim \varepsilon^{-1}$.

Suppose that \eqref{eq:StrichartzCounter} holds true with $I = \R$. Then,
\begin{equation*}
\varepsilon^{d - \frac{2}{p} - \frac{d}{q}} \lesssim \| U f \|_{L_t^p(\R; L_x^q(\R^d))} \lesssim \| f \|_{M^s_{r,t}(\R^d)} \sim \| f \|_{L^r_s(\R^d)} \lesssim \varepsilon^{d- \frac{d}{r}},
\end{equation*}
which requires 
\begin{equation}
\label{eq:IntegrabilityConditionsStrichartzI}
\frac{2}{p} + \frac{d}{q} \leq \frac{d}{r}.
\end{equation}
For $I = [-1,1]$, we find
\begin{equation*}
\varepsilon^{d - \frac{d}{q}} \lesssim \| U f \|_{L_t^p(I;L_x^q(\R^d))} \lesssim \| f \|_{M^s_{r,t}(\R^d)} \sim \| f \|_{L^r_s(\R^d)} \lesssim \varepsilon^{d - \frac{d}{r}}.
\end{equation*}
As $\varepsilon \to 0$, this yields $q \geq r$.

\medskip

The variant of the \textit{anisotropic Knapp-example at high frequencies} rules out gain of derivatives: For $\lambda \in 2^{\mathbb{N}}$ consider
\begin{equation}
\label{eq:HighFrequencyKnapp}
\hat{f}_\lambda(\xi) = \chi_{(\lambda-1,\lambda+1)}(\xi_1) \chi_{(-1,1)}(\xi_2) \ldots \chi_{(-1,1)}(\xi_d).
\end{equation}
We note that
\begin{equation}
\label{eq:NormHighFrequencyKnapp}
\| f_\lambda \|_{M^s_{r,t}(\R^d)} \sim \| f_\lambda \|_{L^r_s(\R^d)} \sim \lambda^s.
\end{equation}

Setting $\varepsilon = \lambda^{-1}$, we find
\begin{equation*}
f_\lambda(x) = \lambda^d g_\varepsilon(\lambda x).
\end{equation*}
Furthermore, we find by change of variables and the computation for the unit frequency anisotropic Knapp-example:
\begin{equation*}
\| U g_\varepsilon(\lambda \cdot) \|_{L_t^p([-1,1];L^q_x(\R^d))} = \lambda^{- \frac{2}{p} - \frac{d}{q}} \| U g_\varepsilon \|_{L_t^p([-\lambda^2;\lambda^2];L_x^q(\R^d))} \gtrsim \lambda^{- d}.
\end{equation*}
By \eqref{eq:NormHighFrequencyKnapp} the validity of \eqref{eq:StrichartzCounter} or \eqref{eq:ModulationSpaceCounter} requires $s \geq 0$.

\medskip

Finally, we examine the \textit{isotropic Knapp-example}, previously inspected in \cite{RogersSeeger2010}. Let $\theta: \R^n \rightarrow \R$ denote a radial function, supported in $\{ 2^{-2} \leq | \xi | \leq 4 \}$ and equal to $1$ on $\{ 2^{-1} \leq |\xi| \leq 2 \}$. We consider the radially symmetric functions
\begin{equation}
\label{eq:FamilyFunctions}
f_\lambda(x) = \big( \frac{\lambda}{2 \pi} \big)^{n} \int \theta(\xi) e^{ i (\lambda \langle x, \xi \rangle - \lambda^2 |\xi|^2) } d\xi.
\end{equation}
By stationary phase, it was computed in \cite[p.~50]{RogersSeeger2010} that
\begin{equation*}
|f_\lambda(x) | \lesssim 1, \quad |x| \gg \lambda: \; |f_{\lambda}(x)| \leq C_N (\lambda^{-1} |x|)^{-N}.
\end{equation*}
Hence,
\begin{equation*}
\| f_\lambda \|_{L_s^r(\R^d)} \lesssim \lambda^{s+ \frac{d}{r}}.
\end{equation*}
Moreover,
\begin{equation*}
\| f_\lambda \|_{M^s_{r,t}(\R^d)} \lesssim \lambda^{s + \frac{d}{t}}.
\end{equation*}
The latter estimate follows as there are $\sim \lambda^d$ unit cubes in the $\lambda$-annulus and all of them give rise to a comparable $L^p$-norm.

Again by (non-)stationary phase, we find the lower bound (cf. \cite[p.~50]{RogersSeeger2010})
\begin{equation}
\label{eq:LowerBoundRadialSymmetricSol}
\big( \int_{1- \lambda^{-2}/10}^1 \| e^{it \Delta} f_\lambda \|^p_{L^q(\R^d)} \big)^{1/p} \gtrsim \lambda^{d - \frac{d}{q} - \frac{2}{p}}.
\end{equation}
Hence, we find for \eqref{eq:StrichartzCounter} to hold:
\begin{equation*}
\lambda^{d - \frac{d}{q} - \frac{2}{p}} \lesssim \| U f_\lambda \|_{L_t^p([-1,1];L_x^q(\R^d))} \lesssim \| f_{\lambda} \|_{L^r_s(\R^d)} \lesssim \lambda^{\frac{d}{r} + s}.
\end{equation*}
As $\lambda \to \infty$, we find
\begin{equation}
\label{eq:DerivativesLpSpacesConsequences}
d - \frac{d}{q} - \frac{2}{p} \leq \frac{d}{r} + s.
\end{equation}
We find the following for \eqref{eq:ModulationSpaceCounter} to be true:
\begin{equation*}
\lambda^{d - \frac{d}{q} - \frac{2}{p}} \lesssim \| U f_\lambda \|_{L_t^p([-1,1];L_x^q(\R^d))} \lesssim \| f_{\lambda} \|_{M^s_{r,t}(\R^d)} \lesssim \lambda^{\frac{d}{t} + s}.
\end{equation*}
Taking $\lambda \to \infty$, we find
\begin{equation}
\label{eq:DerivativesModulationNormsConsequence}
d - \frac{d}{q} - \frac{2}{p} \leq \frac{d}{t} + s.
\end{equation}

We are ready for the proof of Proposition \ref{prop:NecessaryConditions}:
\begin{proof}[Proof of Proposition \ref{prop:NecessaryConditions}]
The claim \eqref{eq:DerivativeModulationSmoothing} follows from the anisotropic Knapp-example at high frequencies and the isotropic Knapp-example \eqref{eq:DerivativesModulationNormsConsequence}.\\
Likewise, \eqref{eq:ConditionsDerivative} follows. The condition $q \geq r$ follows from considering the anisotropic Knapp-example at low frequencies for finite times; the additional integrability condition \eqref{eq:GlobalIntegrabilityConditions} follows from considering the anisotropic Knapp-example globally in time. The proof is complete.
\end{proof}

We give the proof of Corollary \ref{cor:NonexistenceStrichartz}, which asserts non-existence of global Strichartz estimates
\begin{equation}
\label{eq:GlobalStrichartzEstimateCounter}
\| U f \|_{L_t^p(\R;L_x^q(\R^d))} \lesssim \| f \|_{L^r(\R^d)}.
\end{equation}

\begin{proof}[Proof of Corollary \ref{cor:NonexistenceStrichartz}]
In addition to the examples from above, we note the scaling condition
\begin{equation}
\label{eq:ScalingCondition}
\frac{2}{p} + \frac{d}{q} = \frac{d}{r}.
\end{equation}

By considering the isotropic Knapp-example \eqref{eq:DerivativesLpSpacesConsequences}, using \eqref{eq:ScalingCondition}, and assuming \eqref{eq:GlobalStrichartzEstimateCounter}, we find
\begin{equation*}
\lambda^{d-\frac{2}{p} - \frac{d}{q}} = \lambda^{d - \frac{d}{r}} \lesssim \lambda^{\frac{d}{r}}. 
\end{equation*}
For $r > 2$ and $\lambda \to \infty$ this is impossible.
\end{proof}

Lastly, we show Corollary \ref{cor:FixedTimeEstimateModulation}, which asserts the sharp time-dependence in the fixed time estimate in modulation spaces.

\begin{proof}[Proof of Corollary \ref{cor:FixedTimeEstimateModulation}]
For $|t| \lesssim 1$ there is nothing to prove. For $|t| \gg 1$, we consider again a non-trivial Schwartz initial data, radially symmetric, with $\text{supp} \hat{f} \subseteq B(0,1)$.
For this we find by non-stationary phase
\begin{equation*}
\big| \int e^{ix.\xi} e^{it |\xi|^2} \hat{f}(\xi) d\xi \big| \lesssim_N (1+|x|)^{-N}
\end{equation*}
for $|x| \gg |t|$. Moreover,
\begin{equation*}
\big| \int e^{ix.\xi} e^{it |\xi|^2} \hat{f}(\xi) d\xi \big| \gtrsim (1+|t|)^{-\frac{d}{2}}
\end{equation*}
for $|x| \lesssim |t|$ by \cite[Theorem~7.7.5]{Hoermander2003}. Hence, for $|t| \geq 1$,
\begin{equation*}
\| U(t) f \|_{L^p(\R^d)} \gtrsim |t|^{-\frac{d}{2}} |t|^{\frac{d}{p}}.
\end{equation*}
This shows
\begin{equation*}
\| U(t) \|_{M^s_{p,q} \to M^s_{p,q}} \gtrsim \langle t \rangle^{d | \frac{1}{2} - \frac{1}{p} |}
\end{equation*}
for $1 \leq p \leq 2$. By duality, we find the bound for $2 \leq p \leq \infty$.
\end{proof}

\section{$\ell^2$-decoupling implies Strichartz in modulation spaces}
\label{section:DecouplingStrichartzEstimates}
In this section we show Theorem \ref{thm:SmoothingModulationSpaces}. In the remainder of the section let $I = [0,1]$ and $p \geq 2$. Define
\begin{equation}
\label{eq:DecouplingExponent}
s(p,d) = 
\begin{cases}
0, &\quad 2 \leq p \leq \frac{2(d+2)}{d}, \\
\frac{d}{2} - \frac{d+2}{p}, &\quad \frac{2(d+2)}{d} < p \leq \infty.
\end{cases}
\end{equation}
To conclude Theorem \ref{thm:SmoothingModulationSpaces}, it is enough to prove the 
estimates
\begin{align}
\label{eq:q2Estimate}
\| U f \|_{L^p(I \times \R^d)} &\lesssim \| f \|_{M^{s(p,d)+\varepsilon}_{p,2}(\R^d)}, \\
\label{eq:q1Estimate}
\| U f \|_{L^p(I \times \R^d)} &\lesssim \| f \|_{M_{p,1}(\R^d)}.
\end{align}
The remaining estimates follow after frequency localization and H\"older's inequality in the $\ell^q$-spaces.

\eqref{eq:q2Estimate} is a consequence of $\ell^2$-decoupling. After decoupling, this follows via a kernel estimate. The proof of \eqref{eq:q1Estimate} also invokes the kernel estimate. We set
\begin{equation*}
\mathcal{E} f(x,t) = \int_{\R^d} e^{i(x.\xi + t |\xi|^2)} f(\xi) d\xi.
\end{equation*}
Recall the $\ell^2$-decoupling theorem due to Bourgain--Demeter \cite{BourgainDemeter2015}:
\begin{theorem}[$\ell^2$-decoupling for the paraboloid]
\label{thm:DecouplingParaboloid}
Let $\text{supp}(f) \subseteq \{ \xi \, : \, |\xi| \leq 4 \}$. Then, for any $R \geq 1$, we find the following estimate to hold:
\begin{equation*}
\| \mathcal{E} f \|_{L^p(B_{d+1}(0,R))} \lesssim_\varepsilon R^\varepsilon R^{s(p,d)} \big( \sum_{\Box: R^{-\frac{1}{2}}-\text{cube}} \| \mathcal{E} f_\Box \|^2_{L^p(w(B(0,R))} \big)^{\frac{1}{2}}.
\end{equation*}
\end{theorem}
In the above display $w(B(0,R))$ denotes a smooth version of the indicator function on $B(0,R)$ with high polynomial decay off $B(0,R)$; see Subsection \ref{subsection:IntroductionVariableCoefficient} for further explanation. We show that Theorem \ref{thm:DecouplingParaboloid} implies Strichartz estimates in modulation spaces firstly for frequency localized functions:

\begin{proposition}
\label{prop:SmoothingFromDecouplingFrequencyLocalized}
Let $\text{supp}(\hat{f}) \subseteq \{ \xi: \frac{\lambda}{4} \leq |\xi| \leq 4 \lambda \}$. Then,
we find the following estimate to hold
\begin{equation}
\label{eq:StrichartzEstimateDec}
\| U f \|_{L^p(\R^d) \times I)} \lesssim_\varepsilon \lambda^{\varepsilon+s(p,d)} \| f \|_{M_{p,2}}.
\end{equation}
for any $\varepsilon>0$.
\end{proposition}
\begin{proof}
We rescale to unit frequencies
\begin{equation*}
\| U f \|_{L^p(\R^d \times I)} = \lambda^{-\frac{d+2}{p}} \| U g \|_{L^p(\R^d \times [0,\lambda^2])}
\end{equation*}
with $g(x) = f(x/\lambda)$ and $\text{supp} \hat{g} \subseteq B(0,4)$.\\
Next, we cover $\R^d$ with a finitely overlapping family of $\lambda^2$-balls $B \in \mathcal{B}$ to write
\begin{equation*}
\| U g \|^p_{L^p(\R^d \times [0,\lambda^2])} \leq \sum_{B \in \mathcal{B}} \| U g \|^p_{L^p(B \times [0,\lambda^2])}.
\end{equation*}
We use translation invariance to shift the center of $B \in \mathcal{B}$ into the origin:
\begin{equation*}
\| U g \|_{L^p( B \times [0,\lambda^2])} = \| U g_B \|_{L^p(B(0,\lambda^2) \times [0,\lambda^2])}.
\end{equation*}
This is amenable to $\ell^2$-decoupling:
\begin{equation}
\label{eq:ell2DecoupingLocal}
\| U g_B \|_{L^p(B_d(0,\lambda^2) \times [0,\lambda^2])} \lesssim_\varepsilon \lambda^{s(p,d) + \varepsilon}
\big( \sum_{\Box: \lambda^{-1} - \text{cube}} \| U g_{B, \Box} \|^2_{L^p(w_{B_{d+1}(0,\lambda^2))})} \big)^{1/2}
\end{equation}
and by inverting the translation
\begin{equation}
\label{eq:ell2DecouplingLocalII}
\| U g \|_{L^p(B \times [0,\lambda^2])} \lesssim_\varepsilon \lambda^{s(p,d)+\varepsilon} \big( \sum_{\Box: \lambda^{-1} - \text{cube}} \| U g_{\Box} \|_{L^p(w_{B \times [0,\lambda^2]})}^2 \big)^{1/2}.
\end{equation}
We sum \eqref{eq:ell2DecouplingLocalII} over $B \in \mathcal{B}$ in $\ell^p$ and use Minkowski's inequality to find
\begin{equation*}
\| U g \|_{L^p(\R^d \times [0,\lambda^2])} \lesssim_\varepsilon \lambda^{s(p,d) + \varepsilon} \big( \sum_{\Box: \lambda^{-1} - \text{cube}} \| \chi_{\lambda^2}(t) U g_{\Box} \|^2_{L^p(\R^{d+1})} \big)^{1/2}
\end{equation*}
with $\chi_{\lambda^2}$ denoting a rapidly decaying function off $[0,\lambda^2]$.

The claim follows by a fixed-time kernel estimate.
We compute the kernel with $a_\lambda$ denoting the indicator function of the $\lambda^{-1}$-box centered at $\xi_0$:
\begin{equation*}
K(x,t) = \int e^{i(x.\xi + t |\xi|^2)} a_\lambda(\xi) d\xi
\end{equation*}
Via a change of variables and Galilean symmetry, we find
\begin{equation*}
\begin{split}
K(x,t) &= \lambda^{-d} e^{ix.\xi_0} e^{i \frac{t |\xi_0|^2}{\lambda^2}} \int e^{i \big( \frac{x.\xi'}{\lambda} + \frac{2 \xi_0. \xi'}{\lambda} t \big)} e^{it \frac{|\xi|^2}{\lambda^2}} a(\xi) d\xi = \lambda^{-d} e^{ix.\xi_0} e^{i \frac{t |\xi_0|^2}{\lambda^2}} \int e^{i \frac{\xi'.x}{\lambda}} e^{it \frac{|\xi'|^2}{\lambda^2}} a(\xi') d\xi'.
\end{split}
\end{equation*}
By non-stationary phase, we find the following estimate for $|t| \leq \lambda^2$:
\begin{equation*}
|K(x,t)| \lesssim_N \lambda^{-d} (1+ \lambda^{-1} |x|)^{-N}.
\end{equation*}
For $|t| \geq \lambda^2$, we have the rough bound:
\begin{equation*}
|K(x,t)| \lesssim_N \lambda^{-d}
\begin{cases}
 1, \quad &|x| \lesssim |t|/\lambda, \\
 \big( 1 + \lambda^{-1} |x| \big)^{-N}, \quad &|x| \gg |t|/\lambda.
 \end{cases}
\end{equation*}
Hence, we find
\begin{equation*}
\| K(\cdot, t) \|_{L^1(\R^d)} \lesssim \big(1 + \lambda^{-2} |t| \big)^d.
\end{equation*}
%
Integration in time gives
\begin{equation*}
\| \chi_{\lambda^2}(t) U g_\Box \|_{L^p(\R^{d+1})} \lesssim \lambda^{\frac{2}{p}} \| g_\Box \|_{L^p(\R^d)}.
\end{equation*}
We conclude the proof by inverting the change of variables:
\begin{equation*}
\begin{split}
\| U f \|_{L^p(\R^d \times [0,1])} &= \lambda^{s(p,d)+\varepsilon - \frac{d+2}{p}} \big( \sum_{\Box: \lambda^{-1} \text{cube}} \| U g_\Box \|^2_{L^p(w_{B_{d+1}(0,\lambda^2)})} \big)^{\frac{1}{2}} \\
&\lesssim \lambda^\varepsilon \lambda^{s(p,d)} \lambda^{- \frac{d}{p}} \big( \sum_{\Box: \lambda^{-1}-\text{cube}} \| g_\Box \|_{L^p}^2 \big)^{\frac{1}{2}} \lesssim \lambda^{\varepsilon + s(p,d)} \big( \sum_{k \in \Z^d} \| \Box_k f \|_{L^p}^2 \big)^\frac{1}{2}.
\end{split}
\end{equation*}
\end{proof}
In the special case $d=1$, $p=4$ we can remove the derivative loss entirely:
\begin{proposition}
\label{prop:StrichartzEstimateWithoutDerivativeLoss}
Let $\text{supp}(\hat{f}) \subseteq \{ \xi: \frac{\lambda}{4} \leq |\xi| \leq 4 \lambda \}$. Then,
we find the following estimate to hold
\begin{equation}
\label{eq:StrichartzEstimateWithoutDerivativeLoss}
\| U f \|_{L^4(\R \times I)} \lesssim \| f \|_{M_{4,2}}.
\end{equation}
\end{proposition}

The proof crucially relies on the following square function estimate (cf. \cite[Chapter~IX,~§6]{Stein1993}):
\begin{equation}
\label{eq:CordobaFeffermanSquareFunction}
\| e^{it \Delta} f \|_{L_{t,x}^4([0,N^2] \times B(0,N^2))} \lesssim \big\| \big( \sum_\Box |e^{it \Delta} f_\Box|^2 \big)^{1/2} \big\|_{L^4_{t,x}(w_{B_{N^2}})}
\end{equation}
for $\text{supp} (\hat{f}) \subseteq B_1(0,2)$ with $\Box$ ranging over $N^{-1}$-intervals in Fourier space. Substituting \eqref{eq:CordobaFeffermanSquareFunction} for \eqref{eq:ell2DecoupingLocal} in the proof of Proposition \ref{prop:SmoothingFromDecouplingFrequencyLocalized} yields Proposition \ref{prop:StrichartzEstimateWithoutDerivativeLoss}.

\medskip

By Galilean invariance and a related kernel estimate we prove the following:
\begin{proposition}
\label{prop:ModulationMp1Estimate}
Let $2 \leq p \leq \infty$. Then, we find the following estimate to hold:
\begin{equation}
\label{eq:ModulationSpaceMp1Estimate}
\| U f \|_{L^p([-1,1] \times \R^d)} \lesssim \| f \|_{M_{p,1}}.
\end{equation}
\end{proposition}
\begin{proof}
By Minkowski's inequality, it suffices to show
\begin{equation}
\label{eq:UnitLocalizedEstimate}
\| U \Box_k f \|_{L^p([-1,1] \times \R)} \lesssim \| \Box_k f \|_{L^p(\R^d)}.
\end{equation}
By Galilean invariance, we observe with $\hat{g}(\xi) = \hat{f}(\xi+k)$
\begin{equation*}
\| U \Box_k f (t) \|_{L^p} = \| U \Box_0 g(t) \|_{L^p}.
\end{equation*}
Let $\chi \in C^\infty_c(\R^d)$. Clearly, $K(x,t) = \int_{\R^d} \chi(\xi) e^{i(x.\xi + t|\xi|^2)} d\xi$ is uniformly in $L^1(\R^d)$ for $|t| \leq 1$. Thus,
\begin{equation*}
\| U \Box_0 g \|_{L^p([-1,1] \times \R^d)} \lesssim \| \Box_0 g \|_{L^p(\R^d)} \lesssim \| \Box_k f \|_{L^p(\R^d)}.
\end{equation*}
\end{proof}

We can conclude the proof of Theorem \ref{thm:SmoothingModulationSpaces}:
\begin{proof}[Proof of Theorem \ref{thm:SmoothingModulationSpaces}]
(D) is Proposition \ref{prop:ModulationMp1Estimate}. Next, we show (A), (B), and (C) for $q=2$. By the square function estimate ($P_N$ denotes an inhomogeneous Littlewood-Paley decomposition), Minkowski's inequality, and \eqref{eq:StrichartzEstimateDec}, we find
\begin{equation}
\label{eq:q2EstimateFull}
\begin{split}
\| U f \|_{L^p([-1,1] \times \R^d)} &\lesssim \big\| \big( \sum_N |P_N U f|^2 \big)^{\frac{1}{2}} \big\|_{L^p([-1,1] \times \R^d)} \lesssim \big( \sum_N \| P_N U f \|^2_{L^p([-1,1] \times \R^d)} \big)^{\frac{1}{2}} \\
&\lesssim_\varepsilon \big( \sum_N N^{2(s(p,d)+\varepsilon)} \| P_N f \|^2_{M_{p,2}} \big)^{\frac{1}{2}} \lesssim \| f \|_{M^{s(p,d)+\varepsilon}_{p,2}}.
\end{split}
\end{equation}
For $1 \leq q \leq 2$, (A) follows from \eqref{eq:q2EstimateFull} and interpolating with (D) and for $q \geq 2$, we use the embedding \eqref{eq:RegularitySummability}. Likewise, (B) follows for $q \geq 2$ via \eqref{eq:RegularitySummability}. (C) follows from interpolating (B) for $q=2$ with (D). (E) follows from Proposition \ref{prop:StrichartzEstimateWithoutDerivativeLoss} and Stein's square function estimate as in \eqref{eq:q2EstimateFull}.

\end{proof}

\section{Solving the nonlinear Schr\"odinger equation with slowly decaying initial data}
\label{section:LWPNLS}

In the following we solve the nonlinear Schr\"odinger equation
\begin{equation}
\label{eq:CubicNLSAbstract}
\left\{ \begin{array}{cl}
i \partial_t u + \Delta u &= |u|^2 u, \quad (t,x) \in \R \times \R,\\
u(0) &= u_0 \in D
\end{array} \right.
\end{equation}
outside $L^2$-based Sobolev spaces.
\subsection{Local well-posedness of the cubic NLS for slowly decaying initial data}
In this subsection we prove new local well-posedness results. The local results do not take advantage of the defocusing effect, and the results in this section also hold for the focusing equation:
\begin{equation*}
\left\{ \begin{array}{cl}
i \partial_t u + \Delta u &= - |u|^2 u, \quad (t,x) \in \R \times \R, \\
u(0) &= u_0 \in D.
\end{array} \right.
\end{equation*}

 The smoothing estimates are the key ingredient to estimate the homogeneous solution. In the following we use the terminology due to Bejenaru--Tao \cite[Section~3]{BejenaruTao2006}. For further reference, we write \eqref{eq:CubicNLSAbstract} as abstract evolution equation
\begin{equation}
\label{eq:AbstractEvolutionEquation}
u = L(f) + N_3(u,u,u),
\end{equation}
where $u$ takes values in some solution space $S$, $L:D \to S$ is a densely defined linear operator, and the trilinear operator $N_3: S \times S \times S \to S$ is likewise densely defined. As in \cite[Section~3]{BejenaruTao2006}, we refer to \eqref{eq:AbstractEvolutionEquation} as quantitatively well-posed in the spaces $X$, $S$, if the estimates
\begin{align}
\label{eq:AbstractLinearEstimate}
\| L f \|_{S} &\leq C \| f \|_{D}, \\
\label{eq:NonlinearEstimate}
\| N_3(u_1,u_2,u_3) \|_{S} &\leq C \| u_1 \|_S \| u_2 \|_S \| u_3 \|_S
\end{align}
hold true for all $f \in D$, and $u_1,u_2,u_3 \in S$ and some constant $C$. This implies analytic well-posedness (cf. \cite[Theorem~3]{BejenaruTao2006}) and an expression of the solution in terms of its Picard iterates: We define the nonlinear maps $A_m:D \to S$ for $m=1,2,\ldots$ by the recursive formulae
\begin{align*}
A_1 f &= Lf, \\
A_m f &= \sum_{\substack{m_1,m_2,m_3 \geq 1, \\ m_1 + m_2 + m_3 = m }} N_3(A_{m_1} f, A_{m_2} f, A_{m_3} f) \text{ for } m > 1.
\end{align*}
Then we have the homogeneity property
\begin{equation*}
A_m(\lambda f) = \lambda^m A_m(f) \text{ for all } \lambda \in \R, \; m \geq 1 \text{ and } f \in D,
\end{equation*}
and the Lipschitz bound derived from \eqref{eq:AbstractLinearEstimate} and \eqref{eq:NonlinearEstimate}
\begin{equation*}
\| A_m (f) - A_m (g) \|_S \leq \| f - g \|_D C_1^m \big( \| f \|_D + \| g \|_D \big)^{m-1}.
\end{equation*}
Furthermore, we have the absolutely convergent (in $S$) power series expansion
\begin{equation*}
u[f] = \sum_{m=1}^\infty A_m(f)
\end{equation*}
for all $f \in B_D(0,\varepsilon_0)$. In case of \eqref{eq:CubicNLSAbstract}, we have $A_1 f = Lf = (U(t) f)_{t \in \R}$ and for $m > 1$, $A_m = 0$ if not $m = 2j+1$ for some $j \in \mathbb{N}$. $A_{2j+1}$ admits expansion into ternary trees of depth $j$ with $2j+1$ nodes.

To show the linear estimate in Theorem \ref{thm:LWPL4L6} for initial data in $L^p_s(\R)$, we use the following Schr\"odinger smoothing estimates due to Rogers in the special case of one dimension:
\begin{theorem}[{\cite[Theorem~1]{Rogers2008}}]
\label{thm:LpSmoothing}
Let $p \geq 4$. Then, we find the following estimate to hold
\begin{equation}
\label{eq:LpSmoothingEstimate}
\| e^{it \partial_{xx}} u_0 \|_{L^p([-1,1] \times \R)} \lesssim \| u_0 \|_{L_\alpha^p(\R)}
\end{equation}
provided that $\alpha > 2 \big( \frac{1}{2} - \frac{1}{p} \big) - \frac{2}{p}$.
\end{theorem}
Note that the case $p=4$ is not mentioned in \cite[Theorem~1]{Rogers2008}, but follows by interpolating estimates for $p>4$ with the energy estimate
\begin{equation*}
\| e^{it \partial_{xx}} u_0 \|_{L^2([-1,1] \times \R)} \lesssim \| e^{it \partial_{xx}} u_0 \|_{L^\infty([-1,1],L^2(\R))} \lesssim \| u_0 \|_{L^2(\R)}.
\end{equation*}

The linear estimate for initial data in modulation spaces follows from Theorem \ref{thm:SmoothingModulationSpaces}. To show the trilinear estimate, we use inhomogeneous Strichartz estimates. Recall the following inhomogeneous Strichartz estimates for the one-dimensional Schr\"odinger equation (cf. \cite{KeelTao1998,GinibreVelo1979}):
\begin{theorem}
\label{thm:StrichartzSEQ}
Let $q_i$, $p_i \geq 2$ for $i=1,2$ and $\frac{2}{p_i} + \frac{1}{q_i} = \frac{1}{2}$. Then, we find the following estimate to hold:
\begin{equation*}
\| u \|_{L_t^{p_1}([0,T],L_x^{q_1}(\R))} \lesssim \| u(0) \|_{L^2(\R)} + \| (i \partial_t + \partial_x^2) u \|_{L_t^{p'_2}([0,T],L_x^{q_2'}(\R))}.
\end{equation*}
\end{theorem}
We are ready for the proof of Theorem \ref{thm:LWPL4L6}:
\begin{proof}[Proof of Theorem \ref{thm:LWPL4L6}]
In the following we consider $0<T \leq 1$. The claim follows from \cite[Theorem~3]{BejenaruTao2006} once the linear and trilinear estimate are proved. For the linear estimate in $L^p$-spaces, it suffices to prove for $f \in L^4_s(\R)$ or $f \in L^6_s(\R)$
\begin{align}
\label{eq:L4LinearSmoothing} \| L f \|_{L_t^{\frac{24}{7}}([0,T],L^4(\R))} &\lesssim T^{\frac{1}{24}} \| f \|_{L^4_\varepsilon(\R)}, \\
\label{eq:L6LinearSmoothing} \text{ and } \|L f \|_{L_t^3([0,T],L^6(\R))} &\lesssim T^{\frac{1}{6}} \| f \|_{L^6_{\frac{1}{3}+s}(\R)}.
\end{align}
These estimates follow after H\"older in time from the $L^p$-smoothing estimate \eqref{eq:LpSmoothingEstimate}.

For the linear estimates in modulation spaces, we decompose $f = f_1 + f_2$, $f_1 \in M^s_{6,2}(\R)$ and $f_2 \in L^2(\R)$ in case $D=M^s_{6,2}(\R) + L^2(\R)$. It suffices to show
\begin{align}
\label{eq:L4LinearSmoothingModulation} \| L f \|_{L_t^{\frac{24}{7}}([0,T],L_x^4(\R))} &\lesssim T^{\frac{1}{24}} \| f \|_{M_{4,2}(\R)}, \\
\label{eq:L6LinearSmoothingModulation} \text{ and } \|L f \|_{L_t^3([0,T],L_x^6(\R))} &\lesssim T^{\frac{1}{6}} ( \| f_1 \|_{M^s_{6,2}(\R)} + \| f_2 \|_{L^2(\R)}).
\end{align}
Both estimates hold true by Theorem \ref{thm:SmoothingModulationSpaces} applied to $Lf$ in \eqref{eq:L4LinearSmoothingModulation} and $Lf_1$ in \eqref{eq:L6LinearSmoothingModulation} and Strichartz estimates applied to $Lf_2$. The trilinear estimate follows from the estimates
\begin{equation*}
\| \int_0^t e^{i(t-s) \partial_x^2} F(s) ds \|_{L_{t}^8([0,T], L^4_x(\R))} \lesssim \| F \|_{L_{t}^{\frac{8}{7}}([0,T], L_x^{\frac{4}{3}}(\R))}
\end{equation*}
and
\begin{equation*}
\| \int_0^t e^{i(t-s) \partial_x^2} F(s) ds \|_{L_{t}^6([0,T], L_x^6( \R))} \lesssim \| F \|_{L_t^1([0,1],L_x^2(\R))},
\end{equation*}
which are both covered by Theorem \ref{thm:StrichartzSEQ}, and applying H\"older's inequality. Hence, choosing $T=T(\|f \|_D)$, we can apply the contraction mapping principle in $L_t^r([0,T],L^p(\R))$ with $r$ as above.

In the following we focus on initial data $D = M^\varepsilon_{6,2}(\R) + L^2(\R)$ since $D = M_{4,2}(\R)$ is treated by easier means. To prove $u \in C([0,T],M^\varepsilon_{6,2} + L^2)$, it suffices to show $Lf \in C([0,T],M^\varepsilon_{6,2} + L^2)$ and $N_3(u,u,u) \in C([0,T],M^\varepsilon_{6,2} + L^2)$. Let $f = f_1 + f_2$ with $f_1 \in M^\varepsilon_{6,2}$ and $f_2 \in L^2$. By Minkowski's inequality, we find by $U(t) M^\varepsilon_{6,2} = M^\varepsilon_{6,2}$ and $U(t) L^2= L^2$ that
\begin{equation*}
\lim_{t \to 0} \| (Uf)(t) - f \|_{M^\varepsilon_{6,2} + L^2} \leq \limsup_{t \to 0} \| U f_1(t) - f_1 \|_{M^\varepsilon_{6,2}} + \limsup_{t \to 0} \| Uf_2(t) -f_2 \|_{L^2} = 0.
\end{equation*}
The continuity in $M^\varepsilon_{6,2}$ and $L^2$ is a consequence of $(U(t))_{t \in \R}$ a $C_0$-group in both spaces.

For $N_3(u,u,u)$, it suffices to show continuity in $L^2$. By Strichartz estimates, we find
\begin{equation}
\label{eq:N3L2}
\| N_3(u,u,u) \|_{L_t^\infty([0,T], L^2)} \lesssim \| u \|^3_{L_t^r([0,T],L^p)}
\end{equation}
and
\begin{equation*}
\begin{split}
&\quad \big \| \int_0^t e^{i(t-s) \Delta} (|u|^2 u)(s) ds - \int_0^{t + \delta} e^{i((t+\delta) - s)\Delta} (|u|^2 u)(s) ds \big \|_{L^2} \\
&\leq \big \| (e^{i \delta \Delta} - 1) \int_0^t e^{i(t-s) \Delta} (|u|^2 u)(s) ds \big\|_{L^2} + \big \| \int_t^{t+\delta} e^{i((t+\delta) - s) \Delta} (|u|^2 u)(s) ds \big \|_{L^\infty_{\delta \in I} L^2}.
\end{split}
\end{equation*}
For the first term, the limit is zero as $N_3(u,u,u)(t) \in L^2$ by \eqref{eq:N3L2} and $(U(t))_{t \in \R}$ a $C_0$-group in $L^2$. For the second term, we use again Strichartz estimates to find
\begin{equation*}
\big \| \int_{t}^{t+ \delta} e^{i((t+\delta)-s) \Delta} (|u|^2 u)(s) ds \big \|_{L^\infty_{\delta \in I} L^2} \lesssim \| u \|^3_{L_t^r(I,L^p)}.
\end{equation*}
By multilinearity, we see by similar arguments that for differences of solutions
\begin{equation*}
\| u - \tilde{u} \|_{C([0,T],M^\varepsilon_{6,2} + L^2)} \to 0
\end{equation*}
for $\| u(0) - \tilde{u}(0) \|_{M^\varepsilon_{6,2} + L^2} \to 0$ provided that $T = T(\| u(0) \|_{M^\varepsilon_{6,2} + L^2},\| \tilde{u}(0) \|_{M^\varepsilon_{6,2} + L^2})$ is chosen small enough, according to the local existence time in $L_t^r([0,T],L^p)$. The proof is complete.

\end{proof}
We remark that we have some flexibility in the solution space $S=L_t^p([0,T],L_x^p(\R))$. For small initial data, we can likewise iterate in $L_{t,x}^4([0,1] \times \R)$ or $L_{t,x}^6([0,1] \times \R)$. For $p>6$, although the sharp linear estimates are still at disposal, it is not clear how to apply inhomogeneous Strichartz estimates as directly as above. 
\subsection{Improved local results for slowly decaying data}

In this section we consider initial data in $L^p_s(\R)$, $6<p \leq \infty$, and improve the local result due to Dodson--Soffer--Spencer \cite{DodsonSofferSpencer2020} by passing through modulation spaces. In \cite{DodsonSofferSpencer2020} the authors used an iteration in $L^p$-spaces, which is costly in terms of derivatives because the propagator is not bounded in $L^p$-spaces. In \cite{DodsonSofferSpencer2020} local well-posedness is proved in $L^{4n+2}_{s(n)}$, $n \geq 2$, $s(n) = 2n+2$. Note that the number of required derivatives goes to infinity as $n \to \infty$, which is not the case for the improved result in Theorem \ref{thm:ImprovedLocalWellposednessLpSpaces}. For simplicity, we focus on the small data case with $T=1$. Additionally, we give a second proof of local well-posedness in $L^p$-based Sobolev spaces, which makes use of another solution space, but allows to lower the regularity further. 

\medskip



Let $n \geq 2$. The idea in \cite{DodsonSofferSpencer2020} to solve \eqref{eq:CubicNLSAbstract} with initial data in $L^{4n+2}_{s(n)}$ is to split the expansion
\begin{equation*}
u[f] = \sum_{m \geq 1} A_m(f) = L f + N_3(u,u,u)
\end{equation*}
not into linear and nonlinear part, but to consider higher Picard iterates
\begin{equation}
\label{eq:HigherIterates}
\begin{split}
u^0(t) &= L f, \\
u^1(t) &= N_3(u^0,u^0,u^0), \quad u^2(t) = N_3(u^0+u^1,u^0+u^1,u^0+u^1) - u^1(t), \ldots \\
u^j(t) &= N_3(\sum_{k=0}^{j-1} u^{k}, \sum_{k=0}^{j-1} u^k, \sum_{k=0}^{j-1} u^k) - \sum_{k=1}^{j-1} u^{k}(t). \quad (j \geq 2)
\end{split}
\end{equation}
and to prove existence of $v \in S^0([-1,1] \times \R) = L_t^\infty L_x^2 \cap L_t^4 L_x^\infty$, which solves
\begin{equation*}
v = u - \sum_{j=0}^{n-1} u^j.
\end{equation*}
This is equivalent to
\begin{equation}
\label{eq:DifferenceV}
v = N_3(u,u,u) - \sum_{j=1}^{n-1} u^j = N_3(v+\sum_{j=0}^{n-1} u^j, v+\sum_{j=0}^{n-1} u^j, v+\sum_{j=0}^{n-1} u^j) - \sum_{j=1}^{n-1} u^j.
\end{equation}
The key technical aspect is that $u^j$ contains only terms of the form $A_{k}$ with $k \geq 2j+1$. For $u^i$, $i=0,1$, this is clear. For $j \geq 2$, note the following identity
\begin{equation}
\label{eq:HigherPicardIteratesRelation}
\sum_{k=1}^{j-1} u^k = N_3 \big( \sum_{k=0}^{j-2} u^k, \sum_{k=0}^{j-2} u^k, \sum_{k=0}^{j-2} u^k \big).
\end{equation}
For $j=2$, \eqref{eq:HigherPicardIteratesRelation} is immediate from the definition. For $j \geq 3$, rewrite \eqref{eq:HigherPicardIteratesRelation} as
\begin{equation*}
u^{j-1} = N_3 \big( \sum_{k=0}^{j-2} u^k, \sum_{k=0}^{j-2} u^k, \sum_{k=0}^{j-2} u^k \big) - \sum_{k=1}^{j-2} u^k,
\end{equation*}
which corresponds to the definition of $u^{j-1}$ by \eqref{eq:HigherIterates}. Now, to show that $u^j$ contains only terms of the form $A_k$ with $k \geq 2j+1$ for $j \geq 2$, we use \eqref{eq:HigherPicardIteratesRelation} to rewrite
\begin{equation*}
\begin{split}
u^j(t) &= N_3 \big( \sum_{k=0}^{j-1} u^k, \sum_{k=0}^{j-1} u^k, \sum_{k=0}^{j-1} u^k \big) - \sum_{k=0}^{j-1} u^k \\
&= N_3 \big( \sum_{k=0}^{j-1} u^k, \sum_{k=0}^{j-1} u^k, \sum_{k=0}^{j-1} u^k \big) - N_3 \big( \sum_{k=0}^{j-2} u^k, \sum_{k=0}^{j-2} u^k, \sum_{k=0}^{j-2} u^k \big) \\
&= N_3(u^{j-1},u^{j-1},u^{j-1}) + 3 N_3 \big( u^{j-1}, \sum_{k=0}^{j-2} u^k, \sum_{k=0}^{j-2} u^k \big) + 3 N_3\big( u^{j-1}, u^{j-1}, \sum_{k=0}^{j-2} u^k \big).
\end{split}
\end{equation*}

This approach requires to prove estimates for $A_m (f)$ directly. We use the following lemma, which is a consequence of H\"older's and Young's inequality:
\begin{lemma}[{\cite[Theorem~4.3]{Chaichenets2018}}]
\label{lem:HoelderModulationSpaces}
Let $s \geq 0$, $\frac{1}{p} = \frac{1}{p_1} + \frac{1}{p_2}$. Then, we find the following estimate to hold:
\begin{equation*}
\| f_1 f_2 \|_{M^s_{p,1}} \lesssim \| f_1 \|_{M^s_{p_1,1}} \| f_2 \|_{M^s_{p_2,1}}.
\end{equation*}
\end{lemma}

 We have the following:
\begin{lemma}
\label{lem:RegularityPicardIterates}
Let $n \geq 2$ and $m \in \{1,\ldots, 2n-1 \}$. With the above notations, we find
\begin{equation}
\label{eq:PicardIterateLebesgue}
\| A_m f \|_{L_t^\infty L^{\frac{4n+2}{m}}} \lesssim \| f \|^m_{L^{4n+2}_s}
\end{equation}
for $s > 1 - \frac{1}{4n+2}$.
\end{lemma}
\begin{proof}
Let $1 \leq m \leq 2n-1$. We use the embedding $M_{p,1} \hookrightarrow M_{p,p'} \hookrightarrow L^p$ for $2 \leq p \leq \infty$ to argue
\begin{equation*}
\| A_m f \|_{L_t^\infty L^{\frac{4n+2}{m}}} \lesssim \| A_m f \|_{L_t^\infty M_{\frac{4n+2}{m},1}}.
\end{equation*}
 By iterating Lemma \ref{lem:HoelderModulationSpaces}, we obtain
\begin{equation*}
\| A_m f \|_{L_t^\infty M_{\frac{4n+2}{m},1}} \lesssim \| f \|^m_{M_{4n+2,1}} \lesssim_s \| f \|^m_{M^s_{4n+2,4n+2}} \lesssim \| f \|_{L^s_{4n+2}}^m
\end{equation*}
for $s > 1 - \frac{1}{4n+2}$. 
\end{proof}

We find by a similar argument
\begin{equation*}
\begin{split}
\| A_m f \|_{L_t^\infty L_x^\infty} &\lesssim \| A_m f \|_{L_t^\infty M_{\infty,1}} \lesssim  \| f \|_{M_{\infty,1}}^m \\
&\lesssim \| f \|^m_{M_{4n+2,1}} \lesssim \| f \|^m_{L^s_{4n+2}} \text{ for } s > 1 - \frac{1}{4n+2}
\end{split}
\end{equation*}
and remark that the iteration yields
\begin{equation*}
\| A_m f \|_{L_t^\infty L_x^\infty} \lesssim \| f \|^m_{M^{\tilde{s}}_{4n+2,2}} \text{ for } \tilde{s} > \frac{1}{2}.
\end{equation*}
However, $\tilde{s}> \frac{1}{2}$, which regularity suffices for local well-posedness by the algebra property. Hence, the argument is not helpful to improve the local well-posedness theory in modulation spaces, but will be useful in the next section to prove global results.

In the following, for fixed $n \geq 2$, let
\begin{align}
\label{eq:DerivativeLossLp}
\alpha &= 1 - \frac{1}{4n+2}, \\
\label{eq:DerivativeLossMp}
\tilde{\alpha} &= \frac{1}{2}.
\end{align}

A variant of the argument proves the following:
\begin{lemma}
\label{lem:RegularityHigherIterates}
Let $n \geq 2$, $0 \leq j \leq n-1$, and $u^j$ as in \eqref{eq:HigherIterates}. Then, for $\varepsilon >0$, there is $\varepsilon_n \leq 1$ and $\tilde{\varepsilon}_n \leq 1$ such that
\begin{equation*}
\| u^j \|_{L_{t,x}^\infty} + \| u^j \|_{L_t^\infty L^{\frac{4n+2}{2j+1}}} \lesssim \| \langle \partial_x \rangle^{\alpha + \varepsilon} f \|_{L^{4n+2}}
\end{equation*}
holds true provided that $\| \langle \partial_x \rangle^{\alpha + \varepsilon} f \|_{L^{4n+2}} \leq \varepsilon_n$, and
\begin{equation*}
\| u^j \|_{L_{t,x}^\infty} + \| u^j \|_{L_t^\infty L^{\frac{4n+2}{2j+1}}} \lesssim \| \langle \partial_x \rangle^{\tilde{\alpha} + \varepsilon} f \|_{M_{4n+2,2}}
\end{equation*}
provided that $\| \langle \partial_x \rangle^{\tilde{\alpha} + \varepsilon} f \|_{M_{4n+2,2}} \leq \tilde{\varepsilon}_n$.
\end{lemma}

The iteration is the same as in the proof of Lemma \ref{lem:RegularityPicardIterates}, with additional terms estimated with $u^j$ also in $M_{\infty,1}$. We can prove existence of $v$ with the above estimates at hand:
\begin{proposition}
\label{prop:ExistenceV}
Let $\varepsilon>0$, $n \geq 2$, and $\varepsilon_n, \tilde{\varepsilon}_n \leq 1$ as in Lemma \ref{lem:RegularityHigherIterates}. Then, there is a unique $v \in S^0$ satisfying \eqref{eq:DifferenceV}.
\end{proposition}
\begin{proof}
We rewrite \eqref{eq:DifferenceV} modulo order of the arguments in $N_3$ as
\begin{equation*}
\begin{split}
v &= N_3(v,v,v) + 3 N_3(v,v,\sum_{j=0}^{n-1} u^j) + 3 N_3(v,\sum_{j=0}^{n-1} u^j, \sum_{j=0}^{n-1} u^j) \\
 &\qquad + N_3( \sum_{j=0}^{n-1} u^j, \sum_{j=0}^{n-1} u^j, \sum_{j=0}^{n-1} u^j)- \sum_{j=1}^{n-1} u^j.
 \end{split}
\end{equation*}
By Theorem \ref{thm:StrichartzSEQ} and $u^j \in L^\infty_{t,x}$ by Lemma \ref{lem:RegularityHigherIterates}, we find
\begin{align*}
\| N_3(v,v,v) \|_{S^0} \lesssim \| v \|^3_{S^0}, \; \| N_3(v,v,\sum_{j=0}^{n-1} u^j) \|_{S^0} \lesssim \varepsilon \| v \|^2_{S^0}, \; \| N_3(v,\sum_{j=0}^{n-1} u^j,\sum_{j=0}^{n-1} u^j) \|_{S^0} \lesssim \varepsilon^2 \| v \|_{S^0}. 
\end{align*}
We rewrite the last term by \eqref{eq:HigherPicardIteratesRelation} as
\begin{equation*}
N_3(\sum_{j=0}^{n-1} u^j,\sum_{j=0}^{n-1} u^j,\sum_{j=0}^{n-1} u^j) - N_3(\sum_{j=0}^{n-2} u^j, \sum_{j=0}^{n-2} u^j,\sum_{j=0}^{n-2} u^j)
\end{equation*}
and estimate again via Strichartz estimates
\begin{equation*}
\begin{split}
&\quad \| N_3(\sum_{j=0}^{n-1} u^j,\sum_{j=0}^{n-1} u^j,\sum_{j=0}^{n-1} u^j) - N_3(\sum_{j=0}^{n-2} u^j, \sum_{j=0}^{n-2} u^j,\sum_{j=0}^{n-2} u^j) \|_{S^0} \\
&\lesssim \| u^{n-1} \|_{L_t^\infty L^{\frac{4n+2}{2n-1}}} \big( \sum_{j=0}^{n-1} \| u^j \|_{L_t^\infty L_x^{4n+2}} \big)^2 \lesssim \varepsilon^{2n+1}.
\end{split}
\end{equation*}
The claim follows from applying the contraction mapping principle.
\end{proof}
We have proved the following local well-posedness result for slowly decaying data:
\begin{theorem}
\label{thm:SlowlyDecayingData}
Let $\varepsilon > 0$, $n \geq 2$, and $f$, $\varepsilon_n$, and $\tilde{\varepsilon}_n$ as in Proposition \ref{prop:ExistenceV}. Let $\frac{2}{p} + \frac{1}{4n+2} =\frac{1}{2}$. Then, there is $u \in L_t^p([0,1],L^{4n+2}(\R))$, which satisfies \eqref{eq:AbstractEvolutionEquation}. Furthermore, for $\| f_1 \|_{L^{4n+2}_{\alpha + \varepsilon}} + \| f_2 \|_{L^{4n+2}_{\alpha +  \varepsilon}} \leq \varepsilon_n$ or $\| f_1 \|_{M^{\tilde{\alpha} + \varepsilon}_{4n+2,2}} + \| f_2 \|_{M^{\tilde{\alpha} + \varepsilon}_{4n+2,2}} \leq \tilde{\varepsilon}_n$, we have for the corresponding solutions $\| u_1 - u_2 \|_{L^p([0,1],L^{4n+2})} \to 0$ as $\| f_1 - f_2 \|_{L^{4n+2}_{\alpha + \varepsilon}} \to 0$, or $\| f_1 - f_2 \|_{M^{\tilde{\alpha}+\varepsilon}_{4n+2,2}} \to 0$, respectively.

\end{theorem}

\begin{proof}
The claim follows as $v \in L_t^p([0,1],L^{4n+2})$ by Proposition \ref{prop:ExistenceV} and $u^j \in L^\infty_t([0,1],L^{4n+2})$. Hence,
\begin{equation*}
u = \sum_{j=0}^{n-1} u^j + v = Lf + N_3(u,u,u) \in L_t^p([0,1],L^{4n+2}),
\end{equation*}
and the continuity of the data-to-solution mapping follows from multilinearity.
\end{proof}

In the following we consider again initial data in $L^p_s(\R)$, $6<p \leq \infty$, but choose another solution space. We prove the following, which also covers $p= \infty$ with a finite derivative loss:
\begin{theorem}
\label{thm:ImprovedLocalWellposednessLpSpaces}
Let $6 < p \leq \infty$ and $s > 1 - \frac{1}{p}$. Then, we find \eqref{eq:CubicNLSAbstract} to be quantitatively well-posed with $D= L^p_s$ and $S = L_t^\infty([0,1],M_{\infty,1})$.
\end{theorem}
\begin{proof}
The linear estimate \eqref{eq:AbstractLinearEstimate} follows from $\big( e^{it \Delta} \big)_{t \in \R}$ being a $C_0$-group on $M_{\infty,1}$ and the embeddings \eqref{eq:EmbeddingModulationII} and \eqref{eq:EmbeddingModulationIII}:
\begin{equation*}
\| e^{it \Delta} f \|_{L_t^\infty([0,1],M_{\infty,1})} \lesssim \| f \|_{M_{\infty,1}} \lesssim \| f \|_{M_{p,1}} \lesssim_s \| f \|_{M_{p,p}^s(\R)} \lesssim \| f \|_{L^p_s(\R)}.
\end{equation*}
The nonlinear estimate follows from Minkowski's inequality and Lemma \ref{lem:HoelderModulationSpaces}:
\begin{equation*}
\| N_3(u_1,u_2,u_3) \|_{S} = \| \int_0^t e^{i(t-s)\Delta} (u_1 \overline{u}_2 u_3)(s) ds \|_{L_t^\infty M_{\infty,1}} \lesssim \| u_1 \overline{u}_2 u_3 \|_{L_t^1 M_{\infty,1}} \lesssim \prod_{i=1}^3 \| u_i \|_{L_t^\infty M_{\infty,1}} = \prod_{i=1}^3 \| u_i \|_S.
\end{equation*}
Hence, the claim follows from \cite[Theorem~3]{BejenaruTao2006}.
\end{proof}

\subsection{Global well-posedness for slowly decaying initial data}

Next, we show global results in modulation spaces stated in Theorem \ref{thm:GlobalWellposedness}. We use the blow-up alternative due to Dodson--Soffer--Spencer \cite[Section~3]{DodsonSofferSpencer2020}. Since we work with initial data in modulation spaces, for which we have improved homogeneous Strichartz estimates, this requires less Sobolev regularity compared to the $L^p$-case.

Let $w(t) = U(t) f$ denote the first Picard iterate, and let $v(t) = u(t)-w(t)$. Then $v$ satisfies
\begin{equation*}
v(t) = N_3(v+w,v+w,v+w).
\end{equation*}
We have the following blow-up alternative:
\begin{lemma}[Blow-up alternative]
\label{lem:BlowUp}
Let $s>0$, and $f \in M_{4,2}(\R)$. If $T^*$ is maximal such that $u \in L_t^{\frac{24}{7}}([0,T],L^4(\R))$ for $T<T^*$, but $u \notin L_t^{\frac{24}{7}}([0,T^*],L^4(\R))$, then
\begin{equation*}
\lim_{t \to T^*} \| v(t) \|_{L^2} = \infty.
\end{equation*}
The same blow-up alternative holds for $f \in M^s_{6,2}(\R) + L^2(\R)$ in the solution space $L_t^3([0,T],L^6(\R))$.
\end{lemma}
\begin{proof}
We shall only look into the case $f \in M^s_{6,2}(\R) + L^2(\R)$ as the second claim follows by an easy variation. We argue by contradiction. Let $(t_n) \subseteq [0,T^*)$, $t_n \to T^*$, but
\begin{equation}
\label{eq:BoundedInhomogeneousTerm}
\| v(t_n) \|_{L^2(\R)} \leq C.
\end{equation}

 Firstly, letting $f= f_1+ f_2$, $f_1 \in M^s_{6,2}(\R)$ and $f_2 \in L^2(\R)$, we have
\begin{equation*}
\| L f_1(t) \|_{M^s_{6,2}(\R)} \lesssim_{T^*} \| f_1 \|_{M^s_{6,2}(\R)}, \quad \| L f_2(t) \|_{L^2(\R)} \leq \| f_2 \|_{L^2(\R)}
\end{equation*}
Hence, $\| Lf(t) \|_{M^s_{6,2}(\R) + L^2(\R)} \lesssim_{T^*} \| f \|_{M^s_{6,2}(\R) + L^2(\R)}$. But together with \eqref{eq:BoundedInhomogeneousTerm}, this implies that there is a sequence $t_n \to T^*$ with
\begin{equation*}
\| u(t_n) \|_{M^s_{6,2}(\R) + L^2(\R)} \leq \tilde{C}.
\end{equation*}
Since the local existence time of solutions in $L_t^{3}([0,T],L^6_x(\R))$ only depends on $\| f \|_{M^s_{6,2}(\R) + L^2(\R)}$ by Theorem \ref{thm:LWPL4L6}, we see that we can continue the solutions beyond $T^*$. This is a contradiction.
\end{proof}

Hence, for the proof of global well-posedness it suffices to show
\begin{equation*}
\sup_{t \in [0,T]} \| v(t) \|_{L^2} \leq C(T)
\end{equation*}
for some non-decreasing function $C:[0,\infty) \to [0,\infty)$. Let
\begin{align*}
M(v) &= \frac{1}{2} \int |v|^2 dx, \\
E(v) &= \int \frac{1}{2} |v_x|^2 + \frac{1}{4} |v|^4 dx, \\
\tilde{E}(v) &= \int \frac{1}{2} |v_x|^2 + \frac{1}{4} (|v+w|^4 - |w|^4) dx.
\end{align*}
Note that a priori it is not clear that $E(v(t))$ is finite for $t \neq 0$. The following computations are carried out for initial data from a suitable a priori class, say $f \in \mathcal{S}(\R)$. This ensures all quantities to be finite and allows to justify integration by parts arguments. Since we prove bounds depending only on $\| u_0 \|_{M^s_{p,q}}$ for $p,q < \infty$, the arguments are a posteriori justified by density and well-posedness.

For the homogeneous solutions, we observe by the fixed time estimate in modulation spaces and their embedding properties,
\begin{align*}
\| w(t) \|_{M^s_{4,2}} \lesssim \langle t \rangle^{\frac{1}{8}} \| w(0) \|_{M^s_{4,2}}, \quad \| w(t) \|_{L^4} \lesssim \| w(t) \|_{M_{4,4/3}} \lesssim \| w(t) \|_{M^{\frac{1}{4} + \varepsilon}_{4,2}}.
\end{align*}
By Sobolev embedding, we have for any $t \in [0,T]$
\begin{equation}
\label{eq:M4Embeddings}
\| \langle \partial_x \rangle w(t) \|_{L^4} + \| \langle \partial_x \rangle w(t) \|_{L^\infty} \lesssim_T \| w(0) \|_{M^{\frac{3}{2}+\varepsilon}_{4,2}}.
\end{equation}
For $L^6$-based modulation spaces, we find by the same embedding properties
\begin{equation}
\label{eq:M6Embeddings}
\| \langle \partial_x \rangle w(t) \|_{L^\infty} + \| \langle \partial_x \rangle w(t) \|_{L^6} \lesssim_T \| w(0) \|_{M^{\frac{3}{2}+\varepsilon}_{6,2}}.
\end{equation}
Hence, H\"older's inequality and the assumptions on the initial data imply
\begin{equation*}
E \leq C(T) (\tilde{E} + M + 1).
\end{equation*}
By Lemma \ref{lem:BlowUp}, solutions in $L^p([0,T],L^q)$ with $p$, $q$ as in Theorem \ref{thm:GlobalWellposedness} for any $T>0$ follow from the following:
\begin{proposition}
\label{prop:GlobalExistence}
Let $\varepsilon > 0$, $f \in M^{\frac{3}{2}+\varepsilon}_{4,2}$ or $f \in M^{\frac{3}{2}+\varepsilon}_{6,2}$. For all $T>0$, we have
\begin{equation*}
\sup_{t \in [0,T]} M(v(t)) + E(v(t)) \lesssim_T 1.
\end{equation*}
\end{proposition}
Moreover, Theorem \ref{thm:GlobalWellposedness} follows from this proposition by piecing together the local solutions. We turn to its proof:

\begin{proof}[Proof of Proposition \ref{prop:GlobalExistence}]
Let
\begin{equation*}
(f,g) = \Re \int f(x) \overline{g(x)} dx.
\end{equation*}
We aim to bound the quantity $M(v) + \tilde{E}(v) + 1$ by Gr\o nwall's lemma. For the derivative of $M$, we compute
\begin{equation*}
\begin{split}
\partial_t M(v) = (v,v_t) &= (v,-i(v_{xx} + |v+w|^2(v+w)) \\
&=(v,2|v|^2 w + v^2 \overline{w} + 2|w|^2 v+ w^2 \overline{v} + |w|^2 w).
\end{split}
\end{equation*}
By H\"older's inequality, we compute for $\| w(t) \|_{L^4} \lesssim_T 1$
\begin{equation*}
\partial_t M(v) \lesssim_T E(v)^{\frac{3}{4}} + E(v)^{\frac{1}{2}} + E(v)^{\frac{1}{4}} \lesssim_T M(v) + E(v) + 1.
\end{equation*}
Similarly, for $\| w(t) \|_{L^6} + \| w(t) \|_{L^\infty} \lesssim_T 1$, we find
\begin{equation*}
\partial_t M(v) \lesssim_T M(v)^{\frac{1}{2}} E(v)^{\frac{1}{2}} + M(v)^{\frac{1}{2}} + M(v)^{\frac{1}{2}} \lesssim_T M(v) + E(v) + 1.
\end{equation*}

For the derivative of $\tilde{E}$, we compute
\begin{equation*}
\begin{split}
\partial_t \int \frac{1}{2} |v_x|^2 dx = (-v_t,v_{xx}) &= -(v_t, -i v_t + |v+w|^2(v+w)) \\
&= -(v_t,|v+w|^2(v+w))
\end{split}
\end{equation*}
and
\begin{equation*}
\partial_t \int \frac{1}{4} (|v+w|^4 - |w|^4) dx = (v_t+ w_t, |v+w|^2(v+w)) - (w_t, |w|^2 w).
\end{equation*}
Hence,
\begin{equation*}
\partial_t \tilde{E}(v) = (w_t, |v+w|^2(v+w) - |w|^2 w).
\end{equation*}
This expression has total homogeneity four in $v$, $w$. Let $(h_v,h_w)$ denote the homogeneity in $v$, $w$. Then we have to estimate the cases $(1,3)$, $(2,2)$, $(3,1)$. Let the collected terms be denoted by $A_{(h_v,h_w)}$. If \eqref{eq:M4Embeddings} holds true, then applications of H\"older's inequality give
\begin{equation*}
\begin{split}
|A_{(1,3)}| &= |(w_{xx},2|w|^2 v + \bar{v} w^2)| = |(w_x, 2(\overline{w}_x w v + \overline{w} w_x v + |w|^2 v_x + w_x w \overline{v}) + w^2 v_x)| \\
&\lesssim_T E(v)^{\frac{1}{4}} + E(v)^{\frac{1}{2}}, \\
|A_{(2,2)}| &= |(w_{xx},2|v|^2 w+ \overline{w} v^2)| = |(w_x, 2(v\overline{v}_x w + \overline{v} v_x w+ |v|^2 w_x + \overline{w} v v_x) + \overline{w}_x v^2)|\\
&\lesssim_T E(v)^{\frac{3}{4}} + E(v)^{\frac{1}{2}} M(v)^{\frac{1}{2}} + E(v)^{\frac{1}{2}}, \\
|A_{(3,1)}| &= |(w_{xx},|v|^2 v)| = |(w_x, 2|v|^2 v_x + v^2 \overline{v}_x)| \lesssim_T E(v).
\end{split}
\end{equation*}
If \eqref{eq:M6Embeddings} holds true, then by another application of H\"older's inequality, we find
\begin{equation*}
|A_{(1,3)}| \lesssim_T M(v)^{\frac{1}{2}} + E(v)^{\frac{1}{2}}, \quad |A_{(2,2)}| \lesssim_T M(v)^{\frac{1}{2}} E(v)^{\frac{1}{2}} + M(v)^{\frac{1}{2}}, \quad |A_{(3,1)}| \lesssim_T E(v).
\end{equation*}

Consequently,
\begin{equation*}
\partial_t \tilde{E}(v) \lesssim_T (1+M(v)+E(v)) \lesssim_T (1+M(v)+\tilde{E}(v)).
\end{equation*}
Thus,
\begin{equation*}
\partial_t (1+M(v)+\tilde{E}(v)) \leq C(T) (1+M(v) + \tilde{E}(v))
\end{equation*}
and Gr\o nwall's inequality yields
\begin{equation*}
1+M(v) + \tilde{E}(v) \leq e^{\int_0^T C(t) dt}
\end{equation*}
with $C(t) = C(t,\| u_0 \|_{M^{\frac{3}{2}+\varepsilon}_{4,2}})$, or $C(t)=C(t,\| u_0 \|_{M^{\frac{3}{2}+\varepsilon}_{6,2}})$, respectively. Hence,
\begin{equation*}
M(v) + E(v) \lesssim_T 1.
\end{equation*}
\end{proof}

At last, we remark how the argument of \cite[Section~3]{DodsonSofferSpencer2020} also implies global well-posedness for $f \in L^{4n+2}_s$ for $n \geq 2$ and $s> \alpha+1$ (cp. \eqref{eq:DerivativeLossLp}) or $f \in M^{\tilde{s}}_{4n+2,2}$, where $\tilde{s} > \tilde{\alpha}+1$ (cp. \eqref{eq:DerivativeLossMp}). Let $u_l = \sum_{j=0}^{n-1} u^j$, where $u^j$ is given as in \eqref{eq:HigherIterates} and consider $u = v+u_l$. A blow-up alternative remains valid by the construction of the solution via Strichartz estimates. For this purpose we perceive \eqref{eq:DifferenceV} as an initial value problem with data in $L^2$.
\begin{lemma}
\label{lem:BlowUpSlowDecay}
Let $f \in L^{4n+2}_s$ or $f \in M^{\tilde{s}}_{4n+2,2}$ as above, and $u_l \in L_t^p([0,T],L^{4n+2})$, $v \in S^0([0,T])$ be given as in Theorem \ref{thm:ImprovedLocalWellposednessLpSpaces} for $T<T^*$: If $\limsup_{t \to T^*} \| v(t)\|_{L^2} < \infty$, then $u = u_l + v$ can be continued in $L_t^p L_x^{4n+2}$ beyond $T^*$.
\end{lemma}

 Hence, it suffices to prove
\begin{equation*}
\sup_{t \in [0,T]} M(v(t)) + E(v(t)) \lesssim_T 1
\end{equation*}
for global existence (cf. Proposition \ref{prop:GlobalExistence}). The key ingredient for this is (cf. \cite[Eq.~(3.39)]{DodsonSofferSpencer2020})
\begin{equation}
\label{eq:LpEstimateHigherPicardIterates}
\| u_l(t) \|_{L^\infty \cap L^{4n+2}} + \| \partial_x u_l(t) \|_{L^\infty \cap L^{4n+2}} \lesssim_{T^*} 1,
\end{equation}
which follows in the present context from Lemma \ref{lem:RegularityHigherIterates}. At this point, the claim on global well-posedness follows from the algebraic considerations in the proof of \cite[Theorem~7]{DodsonSofferSpencer2020}.


\section{Variable-coefficient decoupling inequalities for non-elliptic Schr\"odinger equations}
\label{section:VariableCoefficientDecoupling}
In this section we prove variable-coefficient decoupling inequalities for elliptic and hyperbolic phase functions. We start with describing the set-up in Subsection \ref{subsection:IntroductionVariableCoefficient} and then carry out the proof in Subsection \ref{subsection:ellipticDecoupling}.
\subsection{Variable-coefficient oscillatory integral operators}
\label{subsection:IntroductionVariableCoefficient}
We consider smooth functions \\ $a \in C^\infty_c(\R^{d+1} \times \R^{d}), \; a=a_1 \otimes a_2$, $0 \leq a_1, a_2 \leq 1$ and $\phi: B^{d+1}(0,1) \times B^{d}(0,1) \rightarrow \R$, which we shall refer to as amplitude and phase function.

 We associate the oscillatory integral operator
\begin{equation}
\label{eq:oscillatoryIntegralOperator}
Tf(t,x) = \int_{\R^{d}} e^{i \phi(t,x,\xi)} a(t,x,\xi) f(\xi) d\xi
\end{equation}
and the rescaled versions
\begin{equation}
\label{eq:rescaledOscillatoryIntegralOperator}
T^\lambda f(t,x) = \int_{\R^{d}} e^{i \lambda \phi(t/\lambda,x/\lambda,\xi)} a(t/\lambda,x/\lambda,\xi) f(\xi) d\xi
\end{equation}
for different classes of phase functions.

Subject of discussion are variable-coefficient generalizations of the phase function
\begin{equation*}
\phi_{hyp}(t,x;\xi) = \langle x, \xi \rangle + \frac{t \langle \xi, I_d^k \xi \rangle}{2}, \quad I^k_d = diag(1,\ldots,1,\underbrace{-1,\ldots,-1}_k), \quad 0 \leq k \leq d/2.
\end{equation*}
Set also $I_d=diag(1,\ldots,1) \in \mathbb{R}^{d \times d}$.\\
In the following we shall always assume that there are at most as many negative eigenvalues as positive eigenvalues, which is no loss of generality since time reversal $t \rightarrow -t$ flips signs.

We define the Gauss map by
\begin{equation}
\label{eq:definitionGaussMap}
G: B^{d+1} \times B^{d} \rightarrow \mathbb{S}^{d}, \quad G(z;\xi) = \frac{G_0(z;\xi)}{|G_0(z;\xi)|}; \quad z=(t,x),
\end{equation}
 where $m \in \mathbb{N}$ and $B_m$ denotes the unit ball in $\R^m$ and
 \begin{equation}
 \label{eq:definitionG0}
 G_0(z;\omega) = \bigwedge_{j=1}^{n} \partial_{\xi_j} \partial_z \phi(z;\xi)
 \end{equation}
with the standard identification $\bigwedge^{d} \R^{d+1} \cong \R^{d+1}$.

We impose the following conditions on the phase function:
\begin{equation*}
\begin{split}
&H1) \quad \text{rank} \; \partial^2_{\xi x} \phi(z;\xi) = d \quad \forall \; (z,\xi) \in B^{d+1} \times B^{d}, \\
&H2) \quad \partial^2_{\xi \xi} \left. \langle \partial_z \phi(z;\xi), G(z;\xi_0) \rangle \right\rvert_{\xi = \xi_0} \text{ is non-degenerate}.
\end{split}
\end{equation*}
$H1)$ is a non-degeneracy condition, and $H2)$ implies that the constant coefficient approximation of $\phi$ is the adjoint Fourier restriction operator (i.e. extension operator) associated to a non-degenerate surface.

Contrary to the constant-coefficient case $\phi_{hyp}$, rescaling $(t,x) \rightarrow (\lambda^2 t,\lambda x), \;  \xi \rightarrow \xi/\lambda$ yields no exact symmetry. Therefore, it is useful to quantify the conditions $H1)$ and $H2)$. Before doing so, we point out the following more precise versions of $H1)$ and $H2)$, which one may assume without loss of generality:
\begin{equation*}
\begin{split}
&H1^\prime) \quad \det \partial^2_{\xi x} \phi(z;\xi) \neq 0 \text{ for all } (z;\xi) \in T \times X \times \Xi = Z \times \Xi; \\
&H2^\prime_{[k]}) \quad \partial_t \partial^2_{\xi \xi} \phi(z;\xi) \text{ is non-degenerate for all } (z;\xi) \in Z \times \Xi \\
&\quad \quad \text{ and has exactly } k \text{ negative eigenvalues.}
\end{split}
\end{equation*}

Here, $T,X,\Xi$ denote balls of radius less or equal to one around the origin.
To reduce from $H1)$ and $H2)$ to the conditions in the above display, one applies a rotation in space-time. This gives $G(0;0) = e_{d+1}$, and then one uses a partition of unity to suitably localize the support. Moreover, the implicit function theorem implies the existence of smooth functions $\Phi$ and $\Psi$ taking values in $X$ and $\Omega$, respectively, such that
\begin{equation}
\label{eq:definitionPsi}
\partial_x \phi(z;\Psi(z;\xi)) = \xi
\end{equation}
and
\begin{equation}
\label{eq:definitionPhi}
\partial_\xi \phi(t,\Phi(t,x;\xi);\xi) = x.
\end{equation}

The first identity allows us to find a graph parametrization $\xi \mapsto (\partial_z \phi(z;\Psi(z;\xi))) = (\xi,(\partial_t \phi)(z;\Psi(z;\xi)))$ for a hypersurface $\Sigma_{z}$ with non-vanishing curvature. From differentiating the second identity we find $\partial_x \Phi(0;0) = \partial^2_{x \xi} \phi(0;0)^{-1}$.\\
Later on, $H1^\prime)$ and $H2^\prime)$ are quantified. It turns out that one can perceive any phase function satisfying $H1^\prime)$ and $H2^\prime)$ after introducing a partition of unity and rescaling as small smooth perturbations of $\phi_{hyp} = \langle x, \xi \rangle + \frac{t \langle \xi, I^d_k \xi \rangle}{2}$.

For $h \in C^2(B^{d}(0,1),\R)$ let the extension operator $E_h$ be given by
\begin{equation*}
E_h f(t,x) = \int_{B^{d}(0,1)} e^{i(x \xi + t h(\xi))} f(\xi) d\xi,
\end{equation*}
where $f \in L^2$, $\text{supp}(f) \subseteq B^d(0,1)$ and define a smooth weight function, which is essentially a characteristic function on some ball $B^{d+1}(\overline{z},R)$, $\overline{z} = (\overline{t},\overline{x})$:
\begin{equation*}
w_{B(\overline{z},R)}(t,x) = (1+R^{-1}|x-\overline{x}|+R^{-1}|t-\overline{t}|)^{-N}
\end{equation*}
for some large integer $N \in \N$, which is fixed later.

We define the decoupled $L^p$-norm for variable coefficient operators for \\
$1 \leq R \leq \lambda$. Let $\mathcal{T}_R$ denotes a finitely overlapping family of $R^{-1/2}$ balls covering $B^d(0,1)$. Set
\begin{equation*}
\Vert T^\lambda f \Vert_{L^{p,R}_{dec}(S)} = \left( \sum_{\tau \in \mathcal{T}_R} \Vert T^\lambda f_{\tau} \Vert^2_{L^p(S)} \right)^{1/2}
\end{equation*}
for $S$ measurable and
\begin{equation}
\label{eq:decouplingParameter}
\alpha(p,k) = 
\begin{cases}
k \left( \frac{1}{4} - \frac{1}{2p} \right), \; &2 \leq p \leq \frac{2(d+2-k)}{d-k}, \\
\frac{d}{4} - \frac{d+2}{2p}, \; &\frac{2(d+2-k)}{d-k} \leq p < \infty.
\end{cases}
\end{equation}

We recall the constant-coefficient $\ell^2$-decoupling theorem proved in \cite{BourgainDemeter2015, BourgainDemeter2017GeneralDecoupling}:
\begin{theorem}{\cite[Theorem~1.2,~p.~280]{BourgainDemeter2017GeneralDecoupling}}
\label{thm:l2DecouplingConstantCoefficients}
Let $R \geq 1$, $N \geq 10$, $ 2 \leq p < \infty, 0 \leq k \leq d/2$, $\alpha(p,k)$ as in \eqref{eq:decouplingParameter} and $h:B^{d}(0,1) \rightarrow \R$ be a $C^2$-function with Hessian $\partial^2_{\xi \xi} h$ having modulus of eigenvalues in $[C^{-1},C]$ for some $C > 0$. Then, we find for $f$ with $supp(f) \subseteq B(0,1)$ the following estimate to hold:
\begin{equation*}
\Vert E_h f \Vert_{L^{p}(w_{B_R})} \lesssim_{C,N,\varepsilon} R^{\alpha(p,k)+\varepsilon} \Vert E_h f \Vert_{L^{p,R}_{dec}(w_{B_R})}
\end{equation*}
provided that $N \geq N(d,p)$.
\end{theorem}
Strictly speaking, this result was proved in \cite{BourgainDemeter2017GeneralDecoupling} only for the hyperboloid $h(\xi) = \sum_{i=1}^d \alpha_i \xi_i^2$. However, the arguments from \cite{PramanikSeeger2007}, which are illustrated in the context of elliptic surfaces in \cite[Section~7]{BourgainDemeter2015}, yield the more general translation invariant case in a straight-forward manner. See also the discussion below.

Originally, decoupling inequalities were studied for the cone by Wolff in \cite{LabaWolff2002,Wolff2000} to make progress on $L^p$-smoothing estimates (cf. \cite{MockenhauptSeegerSogge1992,MockenhauptSeegerSogge1993}) for the wave equation. These estimates were refined (cf. \cite{GarrigosSeeger2009PlateDecompositions,Bourgain2013MomentInequalities}) until the breakthrough result of Bourgain-Demeter (cf. \cite{BourgainDemeter2015,BourgainDemeter2017StudyGuide}) where sharp decoupling inequalities for the paraboloid were proved. Subsequently, the result was generalized to hyperboloids (cf. \cite{BourgainDemeter2017GeneralDecoupling}). These results also give estimates for exponential sums, in particular essentially sharp Strichartz estimates on irrational tori.

The theory was also extended to non-degenerate curves (cf. \cite{BourgainDemeterGuth2016}). As already pointed out in Beltran-Hickman-Sogge \cite{BeltranHickmanSogge2020}, the decoupling theory seems to extend to the variable coefficient case sharply divergent from the $L^p-L^q$-estimates for oscillatory integral operators. In fact, it is well known that there are strictly less estimates admissible in the constant coefficient case due to Kakeya compression (cf. \cite{Bourgain1991OscillatoryIntegralsSeveralVariables,Bourgain1995NewEstimatesOscillatoryIntegrals,Wisewell2005}). 

Our first result is the following extension of Theorem \ref{thm:l2DecouplingConstantCoefficients}:
\begin{theorem}
\label{thm:l2DecouplingVariableCoefficients}
Let $2 \leq p < \infty$, $n,M \in \mathbb{N}$, $0 \leq k \leq d/2$ and $\alpha(p,k)$ like in \eqref{eq:decouplingParameter}. Suppose that $(\phi,a)$ satisfies $H1^\prime)$ and $H2^\prime_{[k]})$. Then, we find the following estimate to hold:
\begin{equation}
\label{eq:variableCoefficientl2Decoupling}
\Vert T^\lambda f \Vert_{L^p(\R^{d+1})} \lesssim_{\varepsilon, \phi, M,a} \lambda^{\alpha(p,k)+\varepsilon} \Vert T^\lambda f \Vert_{L^{p,\lambda}_{dec}(\R^{d+1})} + \lambda^{-M} \Vert f \Vert_2.
\end{equation}
\end{theorem}
For variable-coefficient generalizations of the phase function $\phi_{cone}(t,x;\xi) = \langle x, \xi \rangle + t |\xi|$ associated to the adjoint Fourier restriction problem of the cone this was carried out in \cite{BeltranHickmanSogge2020}. The proof of Theorem \ref{thm:l2DecouplingVariableCoefficients} adapts this general strategy from \cite{BeltranHickmanSogge2020} to prove variable-coefficient decoupling from constant-coefficient decoupling: on small spatial scales the variable-coefficient oscillatory integral operator is well-approximated by a constant-coefficient operator. It is enough to make progress on this small scale because it extends to any scale by means of parabolic rescaling. Moreover, Iosevich--Liu--Xi  \cite{IosevichLiuXi2019} investigated decoupling inequalities for phase functions with a symmetric Carleson--Sj\"olin condition. More recently, Hickman--Iliopoulou \cite{HickmanIliopoulou2020} also showed narrow decoupling inequalities in the variable-coefficient indefinite signature case.

Already in the context of constant coefficients, approximating one surface by another on small scales and recovering arbitrary scales by rescaling was used to derive decoupling estimates for more general elliptic surfaces or the cone (cf. \cite[Section~7,~8]{BourgainDemeter2015}), see also \cite{PramanikSeeger2007,GuoOh2020}.

It seems plausible that a similar approximation derives the variable-coefficient cone decoupling from the variable-coefficient paraboloid decoupling. Recently, in \cite{Harris2019} was shown by the same approximation that broad-narrow considerations are also valid for the cone. We do not pursue this line of argument.

We recall different consequences of Theorem \ref{thm:l2DecouplingVariableCoefficients}: The variable-coefficient $\ell^2$-decoupling implies a stability theorem for exponential sums which is proved using the argument in \cite{BourgainDemeter2015} for the constant-coefficient case. Moreover, on small spatial scales the broad-narrow considerations from the constant-coefficient case extends to the variable- coefficient case. As used above, the decoupling theorem implies Strichartz and smoothing estimates without further arguments (e.g. dispersive estimates for the propagator).
\subsection{Variable-coefficient decoupling for hyperbolic phase functions}
\label{subsection:ellipticDecoupling}
\subsubsection{Basic reductions}
Before we begin the proof of Theorem \ref{thm:l2DecouplingVariableCoefficients} in earnest, we carry out several reductions. Most importantly, we quantify the conditions $H1^\prime)$ and $H2^\prime_{[k]})$.
In dependence of $\varepsilon$, $M$ and $p$ from Theorem \ref{thm:l2DecouplingVariableCoefficients}, we choose a small constant $0 < c_{par} \ll 1$ and a large integer $N=N_{\varepsilon,M,p}$ and define the following conditions which we will impose on the phase function for $A \geq 1$:
\begin{itemize}
\item[$H1_{1})$] $|\partial^2_{x \xi} \phi(z;\xi) - I_d | \leq c_{par}$ for all $(z, \xi) \in Z \times \Xi$,
\item[$H2_{[k]}^{1})$] $|\partial_t \partial_\xi^2 \phi - I_d | \leq c_{par}$ for all $(z,\xi) \in Z \times \Xi$,
\item[$D1^1_{1})$] $\Vert \partial_{x_k} \partial_\xi^{\beta} \phi \Vert_{L^\infty(Z \times \Xi)} \leq c_{par}$ for $2 \leq |\beta| \leq N$,
\item[$D1^2_{1})$] $\Vert \partial_t \partial_\xi^{\beta} \phi \Vert_{L^\infty(Z \times \Xi)} \leq c_{par}$ for $3 \leq |\beta| \leq N$,
\item[$D2_{A})$] $\Vert \partial_z^2 \partial_\xi^\beta \phi \Vert_{L^\infty} \leq \frac{c_{par} A }{100 n}$ for $1 \leq |\beta| \leq 2N$.
\end{itemize}
For technical reasons we also impose the following margin condition on the positional part $a_1$ of the amplitude $a$:
\begin{enumerate}
\item[$M_{A})$] $\text{dist}(\text{supp} a_1, \R^{n+1} \backslash Z) \geq 1/(4 A)$.
\end{enumerate}
We already note the following consequence of $H2_{[k]}^1)$:
\begin{equation}
\label{eq:almostLinearGroupVelocity}
|\partial_t \nabla_\xi \phi | \leq 2 |\xi|.
\end{equation}

 In \cite{GuthHickmanIliopoulou2019} it was shown that after introducing suitable partition of unities and performing changes of variables an elliptic phase function satisfying $H1^\prime)$ and $H2^\prime_{[0]})$ reduces to the following normal form:
\begin{equation}
\label{eq:normalFormEllipticPhaseFunction}
\phi(t,x;\xi) = \langle x, \xi \rangle + \frac{t |\xi|^2}{2} + \mathcal{E}(x,t;\xi)
\end{equation}
with $\mathcal{E}$ being quadratic in $(t,x)$ and $\xi$, to say
\begin{equation*}
\partial_{(x,t)}^\alpha \partial_\xi^\beta \mathcal{E}(0;\xi) = 0 \quad \forall |\alpha| \leq 1, \; \beta \in \mathbb{N}_0^d.
\end{equation*}

The explicit representation \eqref{eq:normalFormEllipticPhaseFunction} is not required for the following arguments. However, it is useful to keep it in mind stressing the nature of a small smooth perturbation to $\phi_{hyp}$.
We refer to data $(\phi,a)$ satisfying the above conditions for some $A \geq 1$ and $0 \leq k \leq d/2$ as type $(A,k)$-data. The notation and nomenclature is analogous to \cite{BeltranHickmanSogge2020} to point out the similarity to the case of homogeneous variable-coefficient phase functions.

It turns out that these conditions are invariant under parabolic rescaling in a uniform sense, and this allows us to run the induction argument for normalized data. However, to reduce arbitrary hyperbolic phase functions, we have to do one rescaling which depends on the phase function. This gives rise to the dependence on $\phi$ in \eqref{eq:variableCoefficientl2Decoupling}. If we confine ourselves in \eqref{eq:variableCoefficientl2Decoupling} to normalized data, there will be no explicit dependence on $\phi$. We define the relevant constant as follows, where $\varepsilon$, $M$ and $p$ were fixed above and $c_{par}$ and $N=N_{\varepsilon,M,p}$ in the definition of normalized data are chosen in dependence.

We denote by $\mathfrak{D}^\varepsilon_{A,k}(\lambda;R)$ the infimum over all $D \geq 0$ so that the estimate
\begin{equation*}
\Vert T^\lambda f \Vert_{L^p(B_R)} \leq D R^{\alpha(p,k)+\varepsilon} \Vert T^\lambda f \Vert_{L^{p,R}_{dec}(w_{B_R})} + R^{2d} (\lambda/R)^{-N/8} \Vert f \Vert_{L^2}
\end{equation*}
holds true for all data $(\phi,a)$ of type $(A,k)$, balls $B_R$ of radius $R$ contained in $B(0,\lambda)$ and $f \in L^2(B^d(0,1))$. For the weight function we take $N$ as in $D2_{A})$. The estimate
\begin{equation}
\label{eq:BoundInduction}
\mathfrak{D}^\varepsilon_{1,k}(\lambda;R) \leq C_\varepsilon
\end{equation}
implies Theorem \ref{thm:l2DecouplingVariableCoefficients} since we can reduce to normal data. It turns out that it is enough to prove the following proposition:
\begin{proposition}
\label{prop:inductionBound}
Let $1 \leq R \leq \lambda^{1-\varepsilon/d}$. Then, we find the estimate \eqref{eq:BoundInduction} to hold true.
\end{proposition}
In fact, we observe that for any $1 \leq \rho \leq R$ and $\rho^{-1/2}$-ball $\theta$ one may write
\begin{equation*}
T^\lambda f_\theta = \sum_{\substack{\sigma \cap \tilde{\theta} \neq \emptyset, \\ \sigma:R^{-1/2}-ball}} T^\lambda f_\sigma,
\end{equation*}
where $\tilde{\theta}$ denotes the intersection of $\text{supp}(f)$ and $\theta$. We compute using Minkowski's and Cauchy-Schwarz inequality that for any weight $w$ one has
\begin{equation}
\label{eq:relationDecouplingScales}
\begin{split}
\Vert T^\lambda f \Vert_{L^{p,\rho}_{dec}(w)} &= ( \sum_{\theta: \rho^{-1/2}-ball} \Vert T^\lambda f_\theta \Vert^2_{L^p(w)} )^{1/2} = ( \sum_{\theta: \rho^{-1/2}-ball} ( \sum_{\substack{\sigma: R^{-1/2}-ball, \\ \sigma \cap \tilde{\theta} \neq \emptyset}} \Vert T^\lambda f_\theta \Vert )^2 )^{1/2} \\
&\leq ( \sum_{\theta: \rho^{-1/2}-ball} (R/\rho)^{d/2} \sum_{\substack{\sigma: R^{-1/2}-ball, \\ \sigma \cap \tilde{\theta} \neq \emptyset}} \Vert T^\lambda f_\sigma \Vert_{L^p}^2 )^{1/2} \lesssim (R/\rho)^{\frac{d}{4}} \Vert T^\lambda f \Vert_{L^{p,R}_{dec}(w)}.
\end{split}
\end{equation}
Since $\Vert T^\lambda f \Vert_{L^p(B_R)} \lesssim \Vert T^\lambda f \Vert_{L^{p,1}_{dec}(B_R)}$, from taking $\rho = 1$ in the above display it follows that
\begin{equation}
\label{eq:finitenessInductionConstant}
\mathfrak{D}_{A,k}^\varepsilon(\lambda;R) \lesssim R^{\frac{d}{4}-\alpha(p,k)-\varepsilon},
\end{equation}
which yields finiteness of $\mathfrak{D}^\varepsilon$. Moreover, we can reduce to
\begin{equation}
\label{eq:reductionInductionConstant}
\mathfrak{D}_{A,k}^\varepsilon(\lambda;\lambda^{1-\frac{\varepsilon}{d}}) \leq C_\varepsilon.
\end{equation}

Indeed, the support conditions of the amplitude $a$ imply that the support of $T^\lambda f$ is always contained in $B(0,\lambda)$. We cover $B(0,\lambda)$ by an essentially disjoint family of $\lambda^{1 - \frac{\varepsilon}{d}}$-balls
\begin{equation*}
\Vert T^\lambda f \Vert_{L^p(B(0,\lambda))}^p \leq \sum_{B: \lambda^{1-\frac{\varepsilon}{d}}-balls} \Vert T^\lambda f \Vert^p_{L^p(B)},
\end{equation*}
 and using Minkowski's inequality we find
\begin{equation*}
\begin{split}
\Vert T^\lambda f \Vert_{L^p(B)} &\lesssim \mathfrak{D}^\varepsilon_{A,k}(\lambda;\lambda^{1-\frac{\varepsilon}{d}}) \lambda^{\frac{\varepsilon}{4}} (\lambda^{1-\frac{\varepsilon}{d}})^{\alpha(p,k)+\varepsilon}  \big( \sum_{\theta: (\lambda^{1-\frac{\varepsilon}{d}})^{-1/2}-balls} \Vert T^\lambda f_\theta \Vert^2_{L^p(w_B)} \big)^{1/2} \\
&\qquad + (\lambda^{1-\varepsilon/d})^{2(d+1)} \lambda^{-\varepsilon N/8} \Vert f \Vert_2 \\
&\lesssim \mathfrak{D}^{\varepsilon}_{A,k} (\lambda;\lambda^{1-\frac{\varepsilon}{d}}) (\lambda^{1-\frac{\varepsilon}{d}})^{\alpha(p,k)+\varepsilon} \lambda^{\frac{\varepsilon}{4}} \big( \sum_{\theta:\lambda^{-1/2}-balls} \Vert T^\lambda f_\theta \Vert^2_{L^p(w_{B(0,\lambda)})} \big)^{1/2} \\
&\qquad + \lambda^{2(d+1)-\frac{2 \varepsilon (d+1)}{d}} \lambda^{- \varepsilon N/8} \Vert f \Vert_2.
\end{split}
\end{equation*}

For $N$ large enough in dependence of $\varepsilon$, $d$ and $M$ we find \eqref{eq:variableCoefficientl2Decoupling} to hold from \eqref{eq:relationDecouplingScales} for normalized data.
\subsubsection{Rescaling of variable-coefficient phase functions}
We record the following trivial rescaling allowing us to reduce data of type $A$ to data of type $1$:
\begin{lemma}
\label{lem:trivialRescaling}
 For any $A \geq 1$ we find the following estimate to hold:
\begin{equation}
\label{eq:trivialReduction}
\mathfrak{D}^\varepsilon_{A,k}(\lambda;R) \lesssim_A \mathfrak{D}_{1,k}^\varepsilon(\lambda/A;R/A).
\end{equation}
\end{lemma}
\begin{proof}
Let $(\phi,a)$ be a datum in $A$-normal form. We define $\tilde{\phi}(z;\xi) = A \phi(z/A;\xi)$ and amplitude $\tilde{a}(z;\xi) = a(z/A;\xi)$ and observe that $T^\lambda f = \tilde{T}^{\lambda/A} f$. Note the equivalent behaviour of $\phi$ and $\tilde{\phi}$ under one positional derivative. Hence, we find $(\tilde{\phi},\tilde{a})$ to satisfy $H1_1)$, $H2_{[k]}^1)$, $D1^1_1)$, $D1^2_1)$, and the second derivative amounts to an additional factor of $1/A$. Hence, we find $D2_1)$ to hold. The new margin of the new amplitude $\tilde{a}$ has been increased to size $1/4$ and we find $M_{1})$ to hold. This step might require the additional argument of decomposing the amplitude function through a partition of unity and translating each piece, if necessary, to adjust to the enlarged support $A \text{supp}(a)$. This involves a sum over $\mathcal{O}(A^{d+1})$ operators where each is associated to type $1$-data.

Covering $B(0,R)$ with $R/A$-balls yields another factor of $\mathcal{O}(A^{d+1})$, but these pieces can be bounded by $\mathfrak{D}^\varepsilon_{1,k}(\lambda/A;R/A),$ and the proof is complete. Moreover, the form of the error term allows us to summarize the sum over $\mathcal{O}(A^{d+1})$ error terms again as error term.
\end{proof}
Next, we show the following stability result for normalized phase functions under parabolic rescaling\footnote{Here, the term parabolic refers to the rescaling of time by a quadratic factor compared to space and is not restricted to phase functions related to elliptic (parabolic) surfaces.}. This allows us to properly run an induction argument.
\begin{lemma}{[Parabolic rescaling for hyperbolic phase functions]}
\label{lem:parabolicRescalingDecoupling}
Let $2 \leq p < \infty$, $1 \leq \rho \leq R \leq \lambda$, $0 \leq k \leq d/2$ and $\alpha(p,k)$ like in \eqref{eq:decouplingParameter}. Suppose that $(\phi,a)$ satisfies $H1^\prime)$ and $H2^\prime_{[k]})$ and let $T^\lambda$ be the associated oscillatory integral operator. If $g$ is supported in a $\rho^{-1}$-ball and $\rho$ is sufficiently large, then there exists a constant $\overline{C}(\phi) \geq 1$ such that
\begin{equation*}
\begin{split}
\Vert T^\lambda g \Vert_{L^p(w_{B_R})} &\lesssim_{\varepsilon, N,\phi} \mathfrak{D}_{1,k}^\varepsilon(\lambda/\overline{C} \rho^2, R/\overline{C} \rho^2) (R/\rho^2)^{\alpha(p,k)+\varepsilon} \Vert T^\lambda g \Vert_{L^{p,R}_{dec}(w_{B_R})} + R^{2(d+1)} (\lambda/R)^{-N/8} \Vert g \Vert_2.
\end{split}
\end{equation*}
\end{lemma}
\begin{proof}[Proof of Lemma \ref{lem:parabolicRescalingDecoupling} for phase functions of type $1$]
Let $\xi_0 \in B^d(0,1)$ be the centre of $\rho^{-1}$-ball where $g$ is supported. We perform the change of variables $\xi^\prime = \rho(\xi - \xi_0)$ and we compute
\begin{equation*}
T^\lambda g(z) = \int_{\R^d} e^{i \phi^\lambda(z;\xi)} a^\lambda(z;\xi) g(\xi) d\xi = \int_{\R^d} e^{i \phi^\lambda(z;\xi_0 + \rho^{-1} \xi^\prime)} a^\lambda(z;\xi_0 + \rho^{-1} \xi^\prime) \underbrace{\rho^{-d} g(\xi_0 + \rho^{-1} \xi^\prime)}_{\tilde{g}(\xi^\prime)} d \xi^\prime.
\end{equation*}
We expand $\phi$ to find
\begin{equation*}
\phi(z;\xi_0 + \xi^\prime/\rho) = \phi(z;\xi_0) + [\nabla_{\xi} \phi(z;\xi_0)] \frac{\xi^\prime}{\rho} + \rho^{-2} \int_0^1 (1-r) \langle \partial^2_{\xi \xi} \phi(z;\xi_0 + r \xi^\prime/\rho) \xi^\prime, \xi^\prime \rangle dr.
\end{equation*}

Let $\Phi_{\xi_0}(t,x) = (t,\Phi(t,x;\xi_0)); \quad \Phi^\lambda(t,x) = \lambda \Phi_{\xi_0}(t/\lambda,x/\lambda)$ and we introduce the dilations $D_\rho(t,x) = ( \rho^2 t,\rho x)$ and $D^\prime_{\rho^{-1}}(x) = \rho^{-1}x$. We find 
\begin{equation}
\label{eq:relationPhaseFunctions}
e^{i \lambda \phi(\Phi_{\xi_0}(\rho^2 t/\lambda,\rho x/\lambda);\xi_0)} T^\lambda g \circ \Phi^\lambda_{\xi_0} \circ D_\rho = \tilde{T}^{\lambda/\rho^2} \tilde{g},
\end{equation}
where
\begin{equation*}
\tilde{T}^{\lambda/\rho^2} \tilde{g}(t,x) = \int_{\R^d} e^{i \tilde{\phi}^{\lambda/\rho^2}(t,x;\xi)} \tilde{a}^{\lambda/\rho^2}(x;\xi) \tilde{g}(\xi) d \xi,
\end{equation*}
and the phase $\tilde{\phi}(t,x;\xi)$ is given by
\begin{equation*} 
\langle x, \xi \rangle + \int_0^1 (1-r) \langle \partial^2_{\xi \xi} \phi(\Phi_{\xi_0}(t,D^\prime_{\rho^{-1}} x); \xi_0 + r \xi/\rho) \xi, \xi \rangle dr,
\end{equation*}
and the amplitude $\tilde{a}(y,t;\xi) = a(\Phi_{\xi_0}(t;D^\prime_{\rho^{-1}}y); \xi_0 + \xi/\rho)$.

We verify \eqref{eq:relationPhaseFunctions}: From the definition
\begin{equation*}
(\Phi_{\xi_0}^\lambda \circ D_\rho)(t,x) = \lambda \Phi_{\xi_0}(\rho^2 t / \lambda,\rho x/\lambda)
\end{equation*}
and
\begin{equation*}
\begin{split}
&\quad \phi^\lambda(\Phi_{\xi_0}^\lambda(D_\rho(t,x)),\xi_0+\xi/\rho) = \lambda \phi(\Phi_{\xi_0}(\rho^2 t/\lambda,\rho x/\lambda);\xi_0+\xi/\rho) \\
&\rightarrow \lambda \phi(\Phi_{\xi_0}(\rho^2 t/\lambda,\rho x/\lambda);\xi_0) + \lambda [\nabla_\xi \phi(\Phi_{\xi_0}(\rho^2 t/\lambda,\rho x/\lambda);\xi_0)] \frac{\xi}{\rho} \\
&\quad + \rho^{-2} \lambda \int_0^1 (1-r) \langle \partial^2_{\xi \xi} \phi(\Phi_{\xi_0}(\rho^2 t/\lambda,\rho x/\lambda);\xi_0 + r \rho^{-1} \xi) \xi, \xi \rangle dr,
\end{split}
\end{equation*}
which proves \eqref{eq:relationPhaseFunctions}. If $\phi$ is in normal form, then we can also write
\begin{equation}
\label{eq:normalFormPhiII}
\tilde{\phi}(t,x;\xi) = \langle x, \xi \rangle + \frac{t |\xi|^2}{2} + \int_0^1 (1-r) \langle \partial^2_{\xi \xi} \mathcal{E}(\Phi_{\xi_0}(t,D^\prime_{\rho^{-1}} x), {\xi_0} + r\xi /\rho) \xi, \xi \rangle,
 \end{equation}
 and with $\tilde{g}$ being supported in $B^d(0,1)$ we can assume that $|\xi| \leq 1$.
 
 A change of spatial variables gives 
\begin{equation*}
\Vert T^\lambda g \Vert_{L^p(B_R)} \lesssim_{\phi} \rho^{\frac{d+2}{p}} \Vert \tilde{T}^{\lambda/\rho^2} \tilde{g} \Vert_{L^p((\Phi_{\xi_0}^\lambda \circ D_\rho)^{-1} (B_R))},
\end{equation*}
where the implicit constant stems from the Jacobian of $\Phi_{\xi_0}$, which is controlled by property $D1_{1})$. Note that the implicit constant can be chosen constant for data of type $1$ provided that $c_{par}>0$ is chosen small enough. We cover $(\Phi^\lambda_{\xi_0} \circ D_\rho)^{-1} (B_R)$ with essentially disjoint $R/\rho^2$-balls, $B_{R/\rho^2} \in \mathcal{B}_{R/\rho^2}$ and find
\begin{equation*}
\Vert T^\lambda g \Vert_{L^p(B_R)} \lesssim_\phi \rho^{(d+2)/p}
 \big( \sum_{B_{R/\rho^2} \in \mathcal{B}_{R/\rho^2}} \Vert \tilde{T}^{\lambda/\rho^2} \tilde{g} \Vert^p_{L^p(B_{R/\rho^2})} \big)^{1/p}.
\end{equation*}

We argue below that
\begin{equation}
\label{eq:smallScaleDecoupling}
\begin{split}
\Vert \tilde{T}^{\lambda/\rho^2} \tilde{g} \Vert_{L^p(B_{R/\rho^2})} &\lesssim_{\varepsilon,N} \mathfrak{D}_{1,k}^\varepsilon(\lambda/\overline{C} \rho^2, R/\overline{C} \rho^2) (R/\rho^2)^{\alpha(p,k)+\varepsilon} \Vert \tilde{T}^{\lambda/\rho^2} \tilde{g} \Vert_{L^{p,R/\rho^2}_{dec}(w_{B_{R/\rho^2}})} \\
&\quad + (R/\rho^2)^{2(d+1)} (\lambda/R)^{-N/8} \Vert g \Vert_{L^2(\R^d)}
\end{split}
\end{equation}
holds for each $B_{R/\rho^2} \in \mathcal{B}_{R/\rho^2}$ and some $\overline{C} \geq 1$.

If $(\tilde{\phi},\tilde{a})$ was a type-$1$ datum, this would be a consequence of the definitions.
First, we show how to conclude the proof with \eqref{eq:smallScaleDecoupling}: we can write
\begin{equation*}
\cup_{B_{R/\rho^2} \in \mathcal{B}_{R/\rho^2}} B_{R/\rho^2} \subseteq (\Phi^\lambda_{\xi_0} \circ D_\rho)^{-1} (B_{C_{\phi} R}) = C_{R^\prime},
\end{equation*}
where $B_{C_\phi R}$ is a ball concentric to $B_R$, but with enlarged radius $C_\phi R$ for some $C_\phi \geq 1$ because $\Phi_{\xi_0}$ is a diffeomorphism.\\
Hence, we find from summing the $p$th power on both sides over $R/\rho^2$ balls and inverting the change of variables
\begin{equation*}
\begin{split}
&\quad \mathfrak{D}^\varepsilon_{1,k}(\lambda/\overline{C} \rho^2, R/\overline{C} \rho^2) (R/\rho^2)^{\alpha(p,k)+\varepsilon} \big( \sum_{B_{R/\rho^2} \in \mathcal{B}} \Vert \tilde{T}^{\lambda/\rho^2} \tilde{g} \Vert^p_{L^{p,R/\rho^2}_{dec}(w_{B_{R/\rho^2}})} \big)^{1/p} \\
&\leq \mathfrak{D}^\varepsilon_{1,k}(\lambda/\overline{C} \rho^2,R/\overline{C} \rho^2) (R/\rho^2)^{\alpha(p,k)+\varepsilon} \Vert \tilde{T}^{\lambda/\rho^2} \tilde{g} \Vert_{L^{p,R/\rho^2}_{dec}(w_{C_{R^\prime}})}.
\end{split}
\end{equation*}

Inverting the change of coordinates yields
\begin{equation*}
\begin{split} 
\Vert T^\lambda g \Vert_{L^p(B_R)} &\lesssim_{\varepsilon,N,\phi} \mathfrak{D}_{1,k}^\varepsilon(\lambda/\overline{C} \rho^2, R/\overline{C} \rho^2) (R/\rho^2)^{\alpha(p,k)+ \varepsilon} \\
&( \sum_{\tilde{\theta}: (R/\rho^2)^{-1/2}-ball} \Vert T^\lambda g_\theta \Vert^p_{L^p(w_{B_R})} )^{1/p} + R^{2(n+1)} (\lambda/R)^{-N/8} \Vert g \Vert_{L^2(\R^n)}.
\end{split}
\end{equation*}

It is straight-forward to check that the $\theta$, which are the images of $\tilde{\theta}$ under the mapping $\xi \mapsto \rho(\xi - {\xi_0})$, which inverts the change of variables in frequency space, form a cover of the $\text{supp} g$ with $R^{-1/2}$-balls. Note how the error term compensates the decomposition into $R/\rho^2$ balls. In fact, any $R/\rho^2$-ball contributes with $(R/\rho^2)^{2d}$ and there are roughly $\rho^{2(d+1)}$ $R/\rho^2$-balls.

It remains to prove \eqref{eq:smallScaleDecoupling} for each $B_{R/\rho^2} \in \mathcal{B}_{R/\rho^2}$. For this purpose record the following representations of $\tilde{\phi}_L = \tilde{\phi}(t,(L^{-1})^t x;L\xi)$:
\begin{equation}
\label{eq:representationRescaledPhaseFunctionI}
\tilde{\phi}_L(t,x;\xi) = \langle x, \xi \rangle + \int_0^1 (1-r) \langle \partial^2_{\xi \xi} \phi(\Phi_{\xi_0}(t,D^\prime_{\rho^{-1}} \circ L^{-1} x), {\xi_0} + L r \xi/\rho) L \xi, L \xi) dr,
\end{equation}
and from Taylor expansion we find (up to an irrelevant phase factor)
\begin{equation}
\label{eq:representationRescaledPhaseFunctionII}
\tilde{\phi}_L(t,x;\xi) = \rho^2 \phi(t,\Phi_{\xi_0}(t,D^\prime_{\rho^{-1}} \circ L^{-1} x);{\xi_0} + L \xi/\rho).
\end{equation}
$\tilde{\phi}_L$ is an affinely changed version of $\tilde{\phi}$ for some invertible $L$, so that $\partial_t \partial^2_{\xi \xi} \tilde{\phi}_L(0,0;0) = I^k_n$. We perceive $L=diag(\sqrt{\mu_1},\ldots,\sqrt{\mu_n}) \cdot R$, where $R$ is a rotation and $\mu_1,\ldots,\mu_n$ are the eigenvalues of $\partial_t \partial^2_{\xi \xi} \tilde{\phi}$ which is already close to $I_n^k$ quantified by property $H2^{[k]}_1)$.

We verify $H1_{1})$ for $\tilde{\phi}$: Taking an $x$ derivative of the integral term leads to an expression of the kind $\partial_x \partial_{\xi \xi}^2  \phi \cdot \partial_x \Phi_{\xi_0} \cdot \rho^{-1}$.\\
$\partial_x \partial^2_{\xi \xi} \phi$ is controlled by property $D_{1}^1)$ of $\phi$. From the definition of $\Phi_{\xi_0}$ and the chain rule we find $\partial_x \Phi_{\xi_0} = (\partial_x \partial_\xi \phi)^{-1}$. Since $| \partial^2_{x \xi} \phi - I^k_n | \leq c_{par}$, we have $|\partial_x \Phi_{\xi_0}| \leq 2$ and we find the total expression to be of order $c_{par}/\rho$. Note that taking a frequency derivative does not magnify the size.\\
Likewise we verify $D1_{1}^1)$ for $|\beta| = 2$. For higher derivatives in $\xi$ we can argue with the representation \eqref{eq:representationRescaledPhaseFunctionII} and observe that the bounds for large $\rho$ become smaller and smaller since any derivative in $\xi$ gives rise to a factor of $\rho^{-1}$. In this way one checks the validity of $D2_{1})$.

We check $H2^{[k]}_{1})$: For this purpose we write
\begin{equation*}
\partial_t \partial_{\xi}^2 \tilde{\phi}_L(t,x;\xi) - I^k_d = \partial_t \partial_{\xi}^2 \tilde{\phi}_L(t,x;\xi) - \partial_t \partial^2_{\xi \xi} \tilde{\phi}_L(0,0;0)
\end{equation*}
and use the fundamental theorem of calculus. For an additional $\xi$ derivative we find the contribution to be of size $\mathcal{O}(c_{par} \rho^{-1})$. For positional derivatives we use property $D2_{1})$ of $\phi$ to find this contribution to be also much smaller than $c_{par},$ and thus the claim follows.

The only cases of $D2_{1})$ which require additional reasoning to the above arguments are when there are two time derivatives and only one or two frequency derivatives. Else, the smallness is immediate from \eqref{eq:representationRescaledPhaseFunctionII}. In the case of two time derivatives and few frequency derivatives we have to consider combinations $\partial^2_{tt} \partial_{\xi \xi}^2 \mathcal{E}$, $\partial_t \Phi_{\xi_0}$ and $\partial_{tt}^2 \Phi_{\xi_0}$. $\partial_{tt}^2 \partial_{\xi \xi}^2 \phi$ and higher frequency derivatives are controlled by property $D2_{1})$ of $\phi$ and above we have seen that $\partial_t \Phi_{\xi_0}$ is controlled quantitatively through \eqref{eq:almostLinearGroupVelocity} of $\phi$. The control over $\partial_{tt}^2 \Phi_{\xi_0}$ follows from considering one further time derivative:
\begin{equation}
\label{eq:secondDerivativePhi}
\begin{split}
&\quad \partial_{tt} \partial_\xi \phi^\lambda(\Phi^\lambda(t,x;\xi_0);\xi_0) + \partial_t \partial_{x \xi}^2 \phi^\lambda(\Phi^\lambda(t,x;\xi_0)) \partial_t \Phi^\lambda \\
&+ \partial_t \partial^2_{x \xi} \phi^\lambda(\Phi^\lambda(t,x;\xi_0);\xi_0) \partial_t \Phi^\lambda(t,x;\xi_0) + \partial^2_{xx} \partial_\xi \phi^\lambda(\Phi^\lambda(t,x;\xi_0);\xi_0)(\partial_t \Phi^\lambda)^2 \\
&+ \partial^2_{x \xi} \phi^\lambda \partial^2_{tt} \Phi^\lambda = 0.
\end{split}
\end{equation}

Hence, we find $|\partial_{tt}^2 \partial_\xi \tilde{\phi}_L|$, $|\partial_{tt}^2 \partial_{\xi \xi}^2 \tilde{\phi}_L| \leq C$ independent of $\phi$ with dependence only on the parameters in the definition of type $1$ data. After invoking Lemma \ref{lem:trivialRescaling} with some constant independent of $\phi$ provided that $\phi$ is a datum of type $1$, the proof is complete.
\end{proof}
Finally, we deal with the case of a general phase function. The proof is essentially a reprise of the proof of Lemma \ref{lem:parabolicRescalingDecoupling}. However, the implicit constants are now allowed to depend on $\phi$, and since we are not dealing with a normalized datum from the beginning, the constants may become arbitrarily large.
\begin{proof}
First, we use the trivial rescaling from Lemma \ref{lem:trivialRescaling} $\phi \rightarrow \phi^A = A \phi(z/A,\xi)$ to ensure that
\begin{equation*}
\Vert \partial_z^2 \partial_\xi^\beta \phi \Vert_{L^\infty} \leq \frac{c_{par}}{100 d A} \; \text{ for } |\beta| =1,2.
\end{equation*}
Later, we shall see how to choose $A=A(\phi)$. Next, we break the support of $g$ into $\rho^{-1}$-balls, and again we will choose $\rho \geq 1$ later in dependence of $\phi$.

We carry out the changes of coordinates from the proof of Lemma \ref{lem:parabolicRescalingDecoupling} and again arrive at the representations
\begin{align}
\label{eq:rescaledRepresentationPhaseI}
\tilde{\phi}^A(t,x;\xi) &= \langle x, \xi \rangle + \int_0^1 (1-r) \langle \partial_{\xi \xi}^2 \phi^A(\Phi^A_{\xi_0}(t,D^\prime_{\rho^{-1}} x);\xi_0+r\xi/\rho) \xi, \xi \rangle dr, \\
\label{eq:rescaledRepresentationPhaseII}
\tilde{\phi}^A(t,x;\xi) &= \rho^2 \phi^A(\Phi^A_{\xi_0}(t,D^\prime_{\rho^{-1}} x);\xi_0 + \xi/\rho),
\end{align}
and we define $\tilde{\phi}^A_L$ analogous to the proof of Lemma \ref{lem:parabolicRescalingDecoupling}. We check $H1_1)$ from \eqref{eq:rescaledRepresentationPhaseI} which shows that
\begin{equation}
\partial^2_{x \xi} \tilde{\phi}^A_L = I_d + \mathcal{O}_\varphi(\rho^{-1}).
\end{equation} 
We also find $\Vert \partial_{x_k} \partial^2_{\xi \xi} \phi \Vert_{L^\infty} = \mathcal{O}_\phi(\rho^{-1})$ also follows from \eqref{eq:rescaledRepresentationPhaseII}. Moreover, for higher order derivatives in $\xi$ we get additional factors of $\rho^{-1}$ which proves property $D1^1_1)$.\\
Likewise, we verify $D1^2_1)$ for sufficiently large $\rho$.

For the proof of $H2^{[k]}_1)$ we write again
\begin{equation}
\partial_t \partial^2_{\xi \xi} \tilde{\phi}^A_L - I^k_d = \partial_t \partial_{\xi \xi}^2 \tilde{\phi}^A_L(t,x;\xi) - \partial_t \partial^2_{\xi \xi } \tilde{\phi}^A_L(0,0;0)
\end{equation}
and estimate the difference deploying the fundamental theorem of calculus. The above arguments already yield $\partial_t \partial^3_{\xi \xi \xi} \tilde{\phi}^A_L = \mathcal{O}_\phi(\rho^{-1})$, for positional derivatives we choose $A=A(\phi)$ large enough, so that $\partial_t \partial_z \partial^2_{\xi \xi} \tilde{\phi}^A_L \leq \frac{c_{par}}{100d}$ and we can also control this contribution. Note that here we also need $|\partial^2_{tt} \Phi^\lambda| = \mathcal{O}_\phi(A^{-1})$ which follows from \eqref{eq:secondDerivativePhi}.

We check $D2_1)$ like in the proof of Lemma \ref{lem:parabolicRescalingDecoupling} after choosing $A=A(\phi)$ sufficiently large.
\end{proof}
\subsubsection{Approximation by extension operators}
Let $(\phi,a)$ be a datum of type $1$ giving rise to the oscillatory integral operator $T^\lambda$ and recall that we assume the amplitude function to be of product type: $a(z;\xi) = a_1(z) a_2(\xi)$. Further, recall that
\begin{equation*}
\xi \mapsto (\nabla_{x,t} \phi^\lambda)(\overline{z};\Psi^\lambda(\overline{z};\xi))
\end{equation*}
is a graph parametrisation of a hypersurface $\Sigma_{\overline{z}}$. Thus, we have
\begin{equation}
\label{eq:parametrisationEllipticHypersurface}
\langle z, (\nabla_{x,t} \phi^\lambda(\overline{z};\Psi^\lambda(\overline{z};\xi)) \rangle = \langle x, \xi \rangle + t h_{\overline{z}}(\xi)
\end{equation}
for all $z = (x,t) \in \R^{d+1}$ with $z/\lambda \in Z$ where $h_{\overline{z}}(\xi) = (\partial_t \phi^\lambda(\overline{z};\Psi^\lambda(\overline{z};\xi)))$.

Moreover, from the definition of $\Psi^\lambda$ we have
\begin{equation}
\label{eq:derivativesPsi}
\begin{split}
\xi &= \partial_x \phi^\lambda(\overline{z};\Psi^\lambda(\overline{z};\xi)), \\
I_d &= \partial^2_{x \xi} \phi^\lambda(\overline{z};\Psi^\lambda(\overline{z};\xi)) (\partial_\xi \Psi^\lambda(\overline{z};\xi)), \\
0 &= \partial^3_{x \xi \xi} \phi(\overline{z};\Psi^\lambda(\xi)) (\partial_\xi \Psi^\lambda(\overline{z};\xi))^2 + \partial^2_{x \xi} \phi^\lambda(\overline{z};\Psi^\lambda(\xi)) \partial_{\xi \xi}^2 \Psi^\lambda(\overline{z};\xi).
\end{split}
\end{equation}
And consequently, we find for $1$-normalized data 
\begin{equation}
\label{eq:propertiesPsiTypeI}
\begin{split}
\partial_\xi \Psi^\lambda(\overline{z};\xi) &\sim I_d, \\
|\partial^2_{\xi \xi} \Psi^\lambda(\overline{z};\xi)| &\ll 1.
\end{split}
\end{equation}

Let $E_{\overline{z}}$ denote the extension operator associated to $\Sigma_{\overline{z}}$ given by
\begin{equation*}
E_{\overline{z}} g(x,t) = \int_{\R^n} e^{i(\langle x, \xi \rangle + t h_{\overline{z}}(\xi))} a_{\overline{z}}(\xi) g(\xi) d\xi,
\end{equation*}
where $a_{\overline{z}}(\xi) = a_2 \circ \Psi^\lambda(\overline{z};\xi) |\det \partial_\xi \Psi^\lambda(\overline{z};\xi) |$.\\
We shall see that on small spatial scales $T^\lambda$ is effectively approximated by $E_{\overline{z}}$ and vice versa. We record the following consequence of dealing with $1$-normalized data:
\begin{lemma}
\label{lem:propertiesEigenvalues}
Let $(\phi,a)$ be a type $1$ datum. Each eigenvalue $\mu$ of $\partial_{\xi \xi}^2 h_{\overline{z}}$ satisfies $|\mu| \sim 1$ on $supp(a_{\overline{z}})$. For elliptic phase functions of type $1$ we have $\mu \sim 1$ on $\text{supp}(a_{\overline{z}})$.
\end{lemma}
\begin{proof}
From the definition of $h_{\overline{z}}$ we find
\begin{equation*}
\begin{split}
\partial_\xi h_{\overline{z}}(\xi) &= (\partial_t \partial_\xi \phi^\lambda(\overline{z};\Psi^\lambda(\overline{z};\xi)) \partial_\xi \Psi^\lambda(\overline{z};\xi), \\
\partial^2_{\xi \xi} h_{\overline{z}}(\xi) &= (\partial_t \partial^2_{\xi \xi} \phi^\lambda(\overline{z};\Psi^\lambda(\overline{z};\xi)) (\partial_\xi \Psi^\lambda(\overline{z};\xi))^2 + \partial_t \partial_{\xi} \phi^\lambda(\overline{z};\Psi^\lambda(\overline{z};\xi)) \partial_{\xi \xi}^2 \Psi^\lambda(\overline{z};\xi),
\end{split}
\end{equation*}
and the claim follows from \eqref{eq:propertiesPsiTypeI}.
\end{proof}
This becomes needful when it comes to applying the constant-coefficient $\ell^2$-decoupling theorem, which we repeated in Theorem \ref{thm:l2DecouplingConstantCoefficients}, because Lemma \ref{lem:propertiesEigenvalues} ensures uniformity of the constant from the decoupling inequality.

In the following we analyze $T^\lambda f(z)$ for $z \in B^{d+1}(\overline{z};K) \subseteq B(0,3\lambda/4)$ and $1 \leq K \leq \lambda^{1/2}$. The containment property can be assumed due to the margin condition. We see that the desired approximation identity holds on this spatial scale: we perform a change of variables $\xi = \Psi^\lambda(\overline{z};\tilde{\xi})$ and expand $\phi^\lambda$ around $\overline{z}$ to find
\begin{equation*}
T^\lambda f(z) = \int_{\R^d} e^{i(\langle z - \overline{z}, \nabla_{x,t} \phi^\lambda(\overline{z};\Psi^\lambda(\overline{z};\tilde{\xi})\rangle+\mathcal{E}^\lambda_{\overline{z}}(z-\overline{z};\tilde{\xi}))} a_1^\lambda(z) a_{\overline{z}}(\tilde{\xi}) f_{\overline{z}}(\tilde{\xi}) d\tilde{\xi},
\end{equation*}
where $f_{\overline{z}} = e^{i \phi^\lambda(\overline{z};\Psi^\lambda(\overline{z};\cdot))} f \circ \Psi^\lambda(\overline{z};\cdot)$ and
\begin{equation*}
\mathcal{E}^\lambda_{\overline{z}}(v;\xi) = \frac{1}{\lambda} \int_0^1 (1-r) \langle (\partial_{zz}^2 \phi)((\overline{z}+rv)/\lambda; \Psi^\lambda(\overline{z};\xi)) v ; v \rangle dr.
\end{equation*}
\begin{lemma}
\label{lem:approximationLemma}
Let $T^\lambda$ be an operator associated to a $1$-normalized datum $(\phi,a)$, $0<\delta \leq 1/2$, $1 \leq K \leq \lambda^{1/2-\delta}$ and $\overline{z}/\lambda \in Z$ so that $B(\overline{z};K) \subseteq B(0,3\lambda/4)$.\\
Then, we find the estimates 
\begin{align}
\label{eq:approximationVariableCoefficients}
\Vert T^\lambda f \Vert_{L^p(w_{B(\overline{z};K)})} &\lesssim_N \Vert E_{\overline{z}} f_{\overline{z}} \Vert_{L^p(w_{B(0;K)})} + \lambda^{-\delta N/2} \Vert f \Vert_2, \\
\label{eq:approximationConstantCoefficients}
\Vert E_{\overline{z}} f_{\overline{z}} \Vert_{L^p(w_{B(0;K)})} &\lesssim_N \Vert T^\lambda f \Vert_{L^p(w_{B(\overline{z};K)})} + \lambda^{-\delta N/2} \Vert f \Vert_2.
\end{align}
to hold provided that $N$ is chosen sufficiently large depending on $d, \delta$ and $p$. Here, the constant $N$ is the same for the weight functions, the conditions on the derivatives $D1^1_{1}), \, D1^2_{1}), \, D2_{1})$ and in the exponent of $\lambda$ in the above estimates. Moreover, in case of sharp cutoff \eqref{eq:approximationVariableCoefficients} becomes
\begin{equation}
\label{eq:approximationVariableCoefficientsSharpCutoff}
\Vert T^\lambda f \Vert_{L^p(B(\overline{z};K))} \lesssim_N \Vert E_{\overline{z}} f_{\overline{z}} \Vert_{L^p(w_{B(0;K)})} + \lambda^{-\delta N/2} \Vert f \Vert_2.
\end{equation} 
\end{lemma}
\begin{proof}
We can replace $f$ by $f \varphi$, where in view of the definition $a_{\overline{z}}$ and the fact that we are dealing with a datum of type $1$, we can assume that $\varphi$ is supported in $[0,2 \pi]^d$. After performing a Fourier series decomposition of $e^{i \mathcal{E}^\lambda_{\overline{z}}(v,\xi)} \varphi(\xi)$, one may write
\begin{equation}
\label{eq:FourierSeriesExpansionPerturbation}
e^{i \mathcal{E}^\lambda_{\overline{z}}(v,\xi)} \varphi(\xi) = \sum_{k \in \mathbb{Z}^d} a_k(v) e^{i \langle k, \xi \rangle},
\end{equation}
where $a_k(v) = \int_{[0,2\pi]^d} e^{-i \langle k, \xi \rangle} e^{i \mathcal{E}^\lambda_{\overline{z}}(v;\xi)} \varphi(\xi) d\xi$.

Since $K \leq \lambda^{1/2}$ we find the favourable bound
\begin{equation*}
\sup_{(v;\xi) \in B(0,K) \times \text{supp} a_{\overline{z}}} | \partial_\xi^\beta \mathcal{E}^\lambda_{\overline{z}}(v;\xi) | \lesssim_N \frac{|v|^2}{\lambda}
\end{equation*}
as long as $\beta \in \N_0^d$ with $1 \leq |\beta| \leq 2N$ by virtue of property $D2_1)$ and the computation in \eqref{eq:derivativesPsi} showing that $|\partial_{\xi}^{\beta} \Psi^\lambda(\overline{z};\xi)| \lesssim 1$ as long as $1 \leq |\beta| \leq 2N$. Consequently, integration by parts yields
\begin{equation*}
|a_k(v)| \lesssim_N (1+|k|)^{-N},
\end{equation*}
whenever $|v| \leq 2 \lambda^{1/2}$. We derive the following pointwise identity from \eqref{eq:FourierSeriesExpansionPerturbation}:
\begin{equation*}
|T^\lambda f(\overline{z}+v)| \leq \sum_{k \in \mathbb{Z}^d} |a_k(v)| |E_{\overline{z}}(f_{\overline{z}} e^{i \langle k, \cdot \rangle})(v)| \lesssim \sum_{k \in \mathbb{Z}^d} (1+|k|)^{-N} |E_{\overline{z}}(f_{\overline{z}} e^{i \langle k, \cdot \rangle}(v) |.
\end{equation*}

We decompose further:
\begin{equation*}
\Vert T^\lambda f \Vert_{L^p(w_{B(\overline{z};K)})} \leq \Vert (T^\lambda f) 1_{B(\overline{z};2\lambda^{1/2})} \Vert_{L^p(w_{B(\overline{z};K)})} + \Vert (T^\lambda f) 1_{\R^{d+1} \backslash B(\overline{z};2 \lambda^{1/2})} \Vert_{L^p(w_{B(\overline{z};K)})}.
\end{equation*}
The second term leads to the error term, that is
\begin{equation}
\label{eq:errorBound}
\Vert (T^\lambda f) \chi_{\R^d \backslash B(\overline{z};2 \lambda^{1/2})} \Vert_{L^p(w_{B(\overline{z};K)})} \lesssim \lambda^{\frac{d}{2p}-\delta(N-(d+2))} \Vert f \Vert_{L^2(\R^n)}.
\end{equation}
In fact, we have $\Vert T^\lambda f \Vert_{L^\infty} \lesssim \Vert f \Vert_2,$ and consequently,
\begin{equation*}
\left( \int_{\R^{d+1}} (1+K^{-1}|x|)^{-(d+2)} |T^\lambda f|^p \right)^{1/p} \lesssim K^{n/p} \Vert f \Vert_{L^2} \lesssim \lambda^{\frac{1}{2p}} \Vert f \Vert_2,
\end{equation*}
and the factor $\lambda^{-\delta(N-(d+2))}$ stems from the additional decay of the weight $(1+K^{-1}|x|)^{-N}$ we are actually considering. This gives \eqref{eq:errorBound}, and since the operator $E_{\overline{z}}$ is translation-invariant,
\begin{equation}
E_{\overline{z}}[e^{i \langle k, \cdot \rangle} g](t,x) = E_{\overline{z}}g(t,x+k) \quad \forall (t,x) \in \R^{d+1} \text{ and } k \in \mathbb{R}^d.
\end{equation}

Minkowski's inequality yields
\begin{equation}
\Vert T^\lambda f 1_{B(\overline{z};2\lambda^{1/2})} \Vert_{L^p(w_{B(\overline{z};K)})} \lesssim_N \sum_{k \in \mathbb{Z}^d} (1+|k|)^{-N} \Vert E_{\overline{z}} f_{\overline{z}} \Vert_{L^p(w_{B((k,0),K)})}.
\end{equation}
Next, observe that
\begin{equation}
\label{eq:weightedConstantCoefficientExpression}
\begin{split}
&\quad \sum_{k \in \mathbb{Z}^d} (1+|k|)^{-N} \Vert E_{\overline{z}} f_{\overline{z}} \Vert_{L^p(w_{B((k,0),K)})} \\
&= \sum_{k \in \mathbb{Z}^d} (1+|k|)^{-\frac{N}{p}} (1+|k|)^{N(\frac{1}{p}-1)} \Vert E_{\overline{z}} f_{\overline{z}} \Vert_{L^p(w_{B((k,0),K)}} \\
&\leq \left[ \sum_{k \in \mathbb{Z}^d} (1+|k|)^{-N} \Vert E_{\overline{z}} f_{\overline{z}} \Vert^p_{L^p(w_{B((k,0),K)}} \right]^{1/p} \left( \sum_{k \in \mathbb{Z}^d} (1+|k|)^{N(\frac{1}{p} - 1)p^\prime} \right)^{1/p^\prime} \\
&= C(d,p,N) \left( \int |E_{\overline{z}} f_{\overline{z}}|^p \sum_{k \in \mathbb{Z}^d} (1+|k|)^{-N} w_{B((k,0),K)} \right)^{1/p} \\
&\lesssim_{n,p} \Vert E_{\overline{z}} f_{\overline{z}} \Vert_{L^p(w_{B(0,K)})}.
\end{split}
\end{equation}
For the ultimate estimate one observes
\begin{equation*}
\sum_{k \in \mathbb{Z}^d} (1+|k|)^{-N} w_{B((k,0),K)} \lesssim w_{B(0,K)}.
\end{equation*}

In order to prove \eqref{eq:approximationConstantCoefficients}, we write
\begin{equation*}
E_{\overline{z}} f_{\overline{z}}(v) = \int_{\R^d} e^{i \phi^\lambda(\overline{z}+v;\Psi^\lambda(\overline{z};\xi)} e^{- \mathcal{E}^\lambda_{\overline{z}}(v;\xi)} a_{\overline{z}}(\xi) f \circ \Psi^\lambda(\overline{z};\xi) d\xi.
\end{equation*}
Again, we insert a smooth cutoff $\varphi(\xi)$ supported in $[0,2\pi]^d$ so that
\begin{equation*}
e^{-i \mathcal{E}^\lambda_{\overline{z}}(v;\xi)} \varphi(\xi) = \sum_{k \in \mathbb{Z}^d} e^{i \langle k, \xi \rangle} b_k(v),
\end{equation*}
where $b_k(v)= \int_{[0,2\pi]^d} e^{-i \langle k, \xi \rangle} e^{-i \mathcal{E}^\lambda_{\overline{z}}(v;\xi)} \varphi(\xi) d\xi$. Once more, integration by parts yields the pointwise bound
\begin{equation*}
|b_k(v)| \lesssim_N (1+|k|)^{-2N},
\end{equation*}
and inverting the change of variables gives
\begin{equation*}
|E_{\overline{z}} f_{\overline{z}}(v) | \lesssim_N \sum_{k \in \mathbb{Z}^d} (1+|k|)^{-2N} |T^\lambda \underbrace{[e^{i \langle k, \partial_z \phi^\lambda(\overline{z},\cdot) \rangle} f]}_{\tilde{f}_k}(\overline{z},v) |.
\end{equation*}

From a similar argument to the one from the proof of \eqref{eq:approximationConstantCoefficients}, we have
\begin{equation}
\label{eq:expansionConstantCoefficients}
\Vert E_{\overline{z}} f_{\overline{z}} \Vert_{L^p(w_{B(0,K)})} \lesssim_N \sum_{k \in \Z^d} (1+|k|)^{-2N} \Vert (T^\lambda \tilde{f}_k) \chi_{B(\overline{z},2 \lambda^{1/2})} \Vert_{L^p(w_{B(\overline{z},K)})} + \lambda^{-\delta N/2} \Vert f \Vert_2.
\end{equation}
The $k = 0$ term is alright because it yields the desired quantity. For the higher order terms we use the estimate \eqref{eq:approximationVariableCoefficients} and \eqref{eq:weightedConstantCoefficientExpression} to conclude
\begin{equation*}
\begin{split}
&\quad \sum_{k \in \mathbb{Z}^d, k \neq 0} (1+|k|)^{-2N} \Vert (T^\lambda \tilde{f}_k) \chi_{B(\overline{z},2 \lambda^{1/2})} \Vert_{L^p(w_{B(\overline{z};K)})}\\
&\lesssim_N 2^{-N} \sum_{k \in \mathbb{Z}^d, k \neq 0} (1+|k|)^{-N} \Vert E_{\overline{z}} f_{\overline{z}} \Vert_{L^p(w_{B((k,0),K)})} \\
&\lesssim_N 2^{-N} \Vert E_{\overline{z}} f_{\overline{z}} \Vert_{L^p(w_{B(0,K)})}.
\end{split}
\end{equation*}
Choosing $N$ large enough depending on $n$ and $p$ this quantity can be absorbed into the lefthandside of \eqref{eq:expansionConstantCoefficients} which yields the claim.
\end{proof}
\subsubsection{Conclusion of the proof}
\begin{proof}[{Proof of Proposition \ref{prop:inductionBound}}]
To show Proposition \ref{prop:inductionBound} for fixed parameters $d$, $\varepsilon$ and $N=N(d,\varepsilon)$, it is enough to prove that
\begin{equation}
\mathfrak{D}_{1,k}^\varepsilon(\lambda;R) \lesssim_\varepsilon 1 \text{ for all } 1 \leq R \leq \lambda^{1-\varepsilon/d}.
\end{equation}
We perform an induction on the radius, and with the base case (small $R$) readily settled, we contend the following induction hypothesis:

There is a constant $\overline{C}_\varepsilon \geq 1$ such that $\mathfrak{D}_{1,k}^\varepsilon(\lambda^\prime;R^\prime) \leq \overline{C}_\varepsilon$ holds for all $1 \leq R^\prime \leq R/2$ and all $\lambda^\prime$ satisfying $R^\prime \leq (\lambda^\prime)^{1-\varepsilon/d}$.

We use the approximation lemma on a small spatial scale and lift the resulting estimates to the correct spatial scales through parabolic rescaling: Let $\mathcal{B}_K$ denote a family of finitely-overlapping $K$-balls covering $B_R$ for some $2 \leq K \leq \lambda^{1/4}$. After breaking $B_R$ into $B(\overline{z};K)$-balls the estimate from Lemma \ref{lem:approximationLemma} implies
\begin{equation}
\label{eq:partitionPositionSpace}
\Vert T^\lambda f \Vert_{L^p(B_R)} \lesssim \big( \sum_{B(\overline{z};K) \in \mathcal{B}_K} \Vert T^\lambda f \Vert^p_{L^p(B(\overline{z};K))} \big)^{1/p} \lesssim \big( \sum_{B(\overline{z};K) \in \mathcal{B}_K} \Vert E_{\overline{z}} f_{\overline{z}} \Vert_{L^p(w_{B(0;K)})}^p \big)^{1/p}.
\end{equation}

We apply the constant-coefficient decoupling theorem (Theorem \ref{thm:l2DecouplingConstantCoefficients}) to each small scale and find after reverting the change of coordinates (again using that we are dealing with $1$-normalized data):
\begin{equation}
\label{eq:partitionFrequencySpace}
\begin{split}
\Vert E_{\overline{z}} f_{\overline{z}} \Vert_{L^p(w_{B(0,K)})} &\lesssim_\varepsilon K^{\varepsilon/2+\alpha(p,k)} \Vert E_{\overline{z}} f_{\overline{z}} \Vert_{L^{p,K}_{dec}(w_{B(0,K)})} \\
&\lesssim K^{\alpha(p,k)+\varepsilon/2} \big( \sum_{\sigma:K^{-1/2}-ball} \Vert T^\lambda f_\sigma \Vert^2_{L^p(w_{B(\overline{z};K)})} \big)^{1/2}  + \lambda^{-N/8} K^{2d} \Vert f \Vert_2. 
\end{split}
\end{equation}
Moreover, this estimate holds uniformly in $\overline{z}$ by virtue of the uniform estimates on the Hessian of $h_{\overline{z}}$ derived in Lemma \ref{lem:propertiesEigenvalues}. We plug \eqref{eq:partitionFrequencySpace} into \eqref{eq:partitionPositionSpace} to find after using Minkowski's inequality:
\begin{equation}
\label{eq:ConstantCoefficientDecouplingConsequence}
\Vert T^\lambda f \Vert_{L^p(B_R)} \lesssim K^{\alpha(p,k) + \varepsilon/2} \left( \sum_{\sigma:K^{-1/2}-ball} \Vert T^\lambda f_\sigma \Vert^2_{L^p(w_{B_R})} \right)^{1/2} + \lambda^{-N/8} K^{2d} R^d \Vert f \Vert_{L^2}.
\end{equation}

Next, apply Lemma \ref{lem:parabolicRescalingDecoupling} to each $T^\lambda f_\sigma$ which gives the estimate
\begin{equation}
\label{eq:parabolicRescalingConsequence}
\begin{split}
\Vert T^\lambda f_\sigma \Vert_{L^p(w_{B_R})} &\leq \mathfrak{D}^\varepsilon_{1,k}(\lambda/(\overline{C}K^2), R/(\overline{C}K^2)) (R/K^2)^{\alpha(p,k)+\varepsilon} \Vert T^\lambda f_\sigma \Vert_{L^{p,R}_{dec}(w_{B_R})} \\
&+ R^{2(d+1)} (\lambda/R)^{-N/8} \Vert f_\sigma \Vert_{L^2(\R^d)}.
\end{split}
\end{equation}
We note that $\mathfrak{D}_{1,k}^\varepsilon(\lambda/(\overline{C}K^2), R/(\overline{C}K^2)) \lesssim_\varepsilon 1$ according to our induction hypothesis. Plugging \eqref{eq:parabolicRescalingConsequence} into \eqref{eq:ConstantCoefficientDecouplingConsequence} gives after applying orthogonality
\begin{equation*}
\begin{split}
\Vert T^\lambda f \Vert_{L^p(B_R)} &\leq C_\varepsilon \overline{C}_\varepsilon K^{\varepsilon/2} (R/K^2)^{\alpha(p,k)+\varepsilon}\left( \sum_{\sigma:K^{-1/2}-ball} \Vert T^\lambda f_\sigma \Vert^2_{L^{p,R}_{dec}(w_{B_R})} \right)^{1/2} \\
&\quad + R^{2(d+1)} (\lambda/R)^{-N/8} \Vert f \Vert_2 \\
&\leq C_\varepsilon \overline{C}_\varepsilon K^{-\varepsilon/2} R^{\alpha(p,k)+\varepsilon} \Vert T^\lambda f \Vert_{L^{p,R}_{dec}(w_{B_R)}} + R^{2(d+1)} (\lambda/R)^{-N/8} \Vert f \Vert_{L^2(\R^d)},
\end{split}
\end{equation*}
and we see that induction closes.
\end{proof}
\begin{proof}[Proof of Theorem \ref{thm:l2DecouplingVariableCoefficients}]
To finish the proof of Theorem \ref{thm:l2DecouplingVariableCoefficients}, we break the support of $f \in L^2(B^d(0,1))$ into $\rho^{-1}$-balls, $\rho = \rho(\phi)$, so that after parabolic rescaling we are dealing with a normalized phase function $\tilde{\phi}$. We can apply Proposition \ref{prop:inductionBound} to $\tilde{\phi}$, and the proof is completed using Lemma \ref{lem:parabolicRescalingDecoupling}.
\end{proof}

\subsection*{Acknowledgements} Funded by the Deutsche Forschungsgemeinschaft (DFG, German Research Foundation) -- Project-ID 258734477 -- SFB 1173. I would like to thank Peer Kunstmann and Friedrich Klaus for helpful discussions on modulation spaces, putting the results into context, and for a remark on proving the global result.

\end{document}